\documentclass[12pt]{amsart}
\pdfoutput=1
\usepackage[headings]{fullpage}
\usepackage{amssymb}
\usepackage{graphicx}
\usepackage{wrapfig}
\usepackage{verbatim}
\usepackage{texdraw}
\usepackage{url}
\usepackage[bookmarks=true,%
    colorlinks=true,%
    linkcolor=blue,%
    citecolor=blue,%
    filecolor=blue,%
    menucolor=blue,%
    urlcolor=blue,%
    breaklinks=true]{hyperref}

\newtheorem{theorem}{Theorem}[section]
\theoremstyle{definition}

\newtheorem{lemma}[theorem]{Lemma}
\newtheorem{definition}[theorem]{Definition}
\newtheorem{remark}[theorem]{Remark}

\newtheorem{conjecture}[theorem]{Conjecture}

\def\BN{\mathbb N}
\def\BZ{\mathbb Z}

\def\BQ{\mathbb Q}
\def\BR{\mathbb R}
\def\BC{\mathbb C}

\def\calA{\mathcal A}

\def\calO{\mathcal O}
\def\calT{\mathcal T}

\def\a{\alpha}

\def\ga{\gamma}

\def\la{\langle}
\def\ra{\rangle}

\def\e{\epsilon}

\def\d{\delta}

\def\th{\theta}

\def\longto{\longrightarrow}

\def\SL{\mathrm{SL}}

\def\z{\zeta}

\def\coeff{\mathrm{coeff}}

\def\be{  \begin{equation} }
\def\ee{  \end{equation} }

\def\diag{\text{diag}}
\def\hb{\hbar}

\def\Av{\mathrm{Av}}
\def\den{\mathrm{den}}
\def\z{\zeta}

\newcommand{\ol}{\overline}

\newcommand{\CZ}{\mathcal{Z}}
\newcommand{\CH}{\mathcal{H}}

\newcommand{\C}{\mathbb{C}}
\newcommand{\Z}{\mathbb{Z}}
\newcommand{\Li}{{\rm Li}}

\newcommand{\ie}{\emph{i.e.}}

\newcommand{\mb}{\mathbf}

\begin{document}


\title[Quantum modularity and complex Chern-Simons theory]{
Quantum modularity and complex Chern-Simons theory}
\author{Tudor Dimofte}
\address{Perimeter Institute for Theoretical Physics, 31 Caroline St N, Waterloo, ON N2J 2Y5, Canada; and \newline\indent
         School of Natural Sciences, Institute for Advanced Study, Princeton, NJ 08540, USA; \newline
         \indent on leave from Dept. of Mathematics, University of California, Davis, CA 95616, USA
        }
\email{tudor@math.ucdavis.edu}
\author{Stavros Garoufalidis}
\address{School of Mathematics \\
         Georgia Institute of Technology \\
         Atlanta, GA 30332-0160, USA \newline
         {\tt \url{http://www.math.gatech.edu/~stavros }}}
\email{stavros@math.gatech.edu}
\thanks{
{\em Key words and phrases: knots, quantum modular forms, quantum modularity
conjecture, cusped hyperbolic manifolds, volume, complex Chern-Simons theory, 
Kashaev invariant, gluing equations, Neumann-Zagier equations, Neumann-Zagier 
datum, Nahm series, hyperbolic geometry, ideal triangulations, 1-loop, 
torsion, quantum dilogarithm, state-integral, perturbation theory, 
Feynman diagrams, formal Gaussian integration, cyclic quantum dilogarithm,
number field.
}
}

\date{November 17, 2015}

\begin{abstract}
The Quantum Modularity Conjecture of Zagier predicts the existence of a formal
power series with arithmetically interesting coefficients that appears in
the asymptotics of the Kashaev invariant at each root of unity. 
Our goal is to construct a power series from a Neumann-Zagier datum
(i.e.,  an ideal triangulation of the knot complement and a geometric 
solution to the gluing equations) and a complex root of unity $\zeta$.
We prove that the coefficients of our series lie in the trace field
of the knot, adjoined a complex root of unity.
We conjecture that our series are those that appear in the 
Quantum Modularity Conjecture and confirm that they match the numerical
asymptotics of the Kashaev invariant (at various roots of unity)
computed by Zagier and the first author. Our construction is motivated 
by the analysis of singular limits in Chern-Simons theory 
with gauge group $SL(2,\C)$ at fixed level $k$, where $\zeta^k=1$.
\end{abstract}

\maketitle

\tableofcontents

\section{Introduction}
\label{sec.intro}

\subsection{Quantum modular forms}
\label{sub.QMF}

Quantum modular forms are fascinating objects introduced by 
Zagier~\cite{Za:QMF}. In the simplest formulation, a quantum modular form is a 
complex valued function $f$ on the set of complex roots of 
unity that comes equipped with a suitable formal power series expansion  
$\phi_\zeta(\hbar) \in \BC[\![\hbar]\!]$ at each complex root of unity $\zeta$.
Usually $f$ is given explicitly.
On the other hand, the power series $\phi_\zeta$,
although uniquely determined by $f$, are not easy to obtain.

One of the most interesting examples of quantum modular forms is conjectured
to be the logarithm of the Kashaev invariant of a knot. This 
is Zagier's Quantum Modularity Conjecture~\cite{Za:QMF}. 
Evidence for this conjecture includes ample numerical computations performed by 
Zagier and the first author~\cite{GZ1} as well as a proof in the case of the
$4_1$ knot~\cite{GZ1}.

Our goal is to construct an explicit formula for 
the power series that appear in the Quantum Modularity Conjecture of a knot.
We highlight some features of our results. 

\noindent
\rm{(a)}
The formulas for our series $\phi_{\gamma,\zeta}(\hbar)$ use as input 
a complex root of unity $\zeta$ and Neumann-Zagier datum $\gamma$, 
i.e., an ideal triangulation of a knot complement and a geometric solution 
to the gluing equations. Such a datum is readily available from {\tt SnapPy}.

\noindent
\rm{(b)}
Our series have arithmetically interesting coefficients;
see Theorems~\ref{thm.1loop} and~\ref{thm.1}.

\noindent
\rm{(c)}
Our series lead to exact computations that match the numerically computed
asymptotics of the Kashaev invariant at various roots of unity. The details
of our computations are given in Section~\ref{sec.compute}. 

\noindent
\rm{(d)}
Our construction is motivated by the analysis of singular limits in 
complex Chern-Simons theory, with gauge group $SL(2,\C)$.
Complex Chern-Simons theory depends on two coupling constants or 
{\em levels} $(k,\sigma)\in (\Z,\C)$ \cite{Witten-cx},
and the series $\phi_{\gamma,\zeta}(\hbar)$ at $\zeta=e^{2\pi i/k}$ is obtained 
by sending  $\sigma\to k$ (while keeping $k$ fixed). We describe this limit 
in greater detail in Section \ref{sec.CCS}, after introducing the definition 
of $\phi_{\gamma,\zeta}(\hbar)$ and proving its number-theoretic properties in 
Sections \ref{sec.def}-\ref{sec.proofs}.

\subsection{Zagier's Quantum Modularity Conjecture}
\label{sub.QMC}

The Kashaev invariant of a knot $K$ is a sequence of complex
numbers $\la K \ra_N$ indexed by a positive integer $N$~\cite{Kashaev95}. 
Murakami-Murakami~\cite{MM} identified the Kashaev invariant with an  
evaluation of the colored Jones polynomial $J_{K,N}(q) \in \BZ[q^{\pm 1}]$ 
colored by the $N$-dimensional irreducible
representation of $\mathfrak{sl}_2$:
$$
\la K \ra_N = J_{K,N}(e^{\frac{2 \pi i}{N}}) \,.
$$
Identifying the set of complex roots of unity with $\BQ/\BZ$, 
the Kashaev invariant can be extended to a translation-invariant function 
on $\BQ$ by
\be
J^0_K: \BQ \longto \BC, \qquad 
\a \mapsto J^0_K(\a)=J_{K,\den(\a)}(e^{2 \pi i \a})
\ee
where $\den(a/c)=c$ when $c>0$ and $a,c$ are coprime. Obviously, we have
$J^0_K(\a)=J^0_K(\a+1)$ for $\a \in \BQ$.

The Quantum Modularity Conjecture~\cite{Za:QMF} predicts for each 
complex root of unity $\zeta$ the existence of a formal power series 
$\phi_{K,\zeta}(\hbar) \in \BC[\![\hbar]\!]$ with the following property. 
Choose any $\begin{pmatrix} a & b \\ c & d \end{pmatrix} \in \SL(2,\BZ)$ 
with $c>0$ such that $\zeta = \exp(\frac{2\pi i a}{c})$. Let $X\to\infty$ 
in a fixed coset of $\BQ/\BZ$, and set $\hbar=2\pi i/(cX+d)$.
Then there is an asymptotic expansion
\be
\label{eq.QMF}
J^0_K\left(\frac{aX+b}{cX+d}\right) \sim J^0_K(X) 
\big(\tfrac{2\pi i}{\hbar}\big)^{3/2} e^{V/(c\hbar)}
\phi_{K,\zeta}(\hbar)
\ee
where $V$ is the complex volume of $K$. Furthermore, the coefficients of 
$\phi_{K,\zeta}(\hbar)$ 
are conjectured to be algebraic integers of the
following form:
\begin{itemize}
\item[(a)]
The coefficients of 
$\phi^+_{K,\zeta}(\hbar)=\phi_{K,\zeta}(\hbar)/\phi_{K,\zeta}(0)$ 
should be elements of the field 
$F_K(\zeta)$, where $F_K$ is the trace field of $K$. 
\item[(b)]
The constant term (under some mild assumptions on $F_K(\zeta)$)
should factor as follows:
\be
\label{eq.phi0factor}
\phi_{K,k}(0) = \phi_{K,1}(0) \sqrt[k]{\varepsilon_K} \beta_{K,k}
\ee
where $\phi_{K,1}(0)^2 \in F_K(\zeta)$, $\varepsilon_K$ is a unit (i.e.,
an algebraic integer of norm $\pm 1$)
in $F_K(\zeta)$ that depends only on the element of the Bloch group
of $K$ (and $\zeta$), and $\beta_{K,k} \in F_K(\zeta)$. 
\end{itemize}
The above unit is studied in~\cite{calegari}. For a detailed 
discussion of the Quantum Modularity Conjecture, we refer the reader 
to~\cite{Za:QMF} and also~\cite{GZ1} where a proof for the case of 
the $4_1$ knot (the simplest hyperbolic knot) is given. 

The Quantum Modularity Conjecture includes the Volume Conjecture of Kashaev
~\cite{Kashaev95} and its refinement to all orders of $1/N$ conjectured by
Gukov~\cite{Gu} and by the second author~\cite{Ga:arithmetic}. Indeed,
when $\left(\begin{smallmatrix} a & b \\ c & d \end{smallmatrix}\right)  
= \left(\begin{smallmatrix} 0 & -1 \\ 1 & 0\end{smallmatrix}\right)$
and $X=N\in \BN$, we obtain that
\be
\label{eq.GVC}
\overline{\la K \ra_N} = J^0_K\left(-\frac{1}{N}\right) \sim 
N^{3/2} e^{N V/(2 \pi i)}\phi_{K,1}\left(\frac{2 \pi i}{N}\right)\,.
\ee

The second author and Zagier numerically computed the
Kashaev invariant and its asymptotics, exposing several coefficients of 
the series $\phi_{K,\zeta}(\hbar)$ for many knots, and giving numerical 
confirmation of the Modularity Conjecture. The results are summarized 
in~\cite{GZ1}.

A definition of the power series $\phi_{K,\zeta}(\hbar)$ in the Quantum 
Modularity Conjecture was missing, though. Motivated by this problem, 
in an earlier publication
~\cite{DG} (inspired by \cite{DGLZ}) the authors assigned to a 
Neumann-Zagier datum $\gamma$
(i.e., to an ideal triangulation of the knot complement, a geometric 
solution to the gluing equations and a flattening) a power series 
$\phi_{\gamma,1}(\hbar)$
and conjectured that it coincides with the series $\phi_{K,1}(\hbar)$. 
One advantage of the series $\phi_{\gamma,1}(\hbar)$ 
is the exact computation of its coefficients using standard {\tt SnapPy}
methods~\cite{snappy}, together with 
 finite sums of Feynman diagrams. 
In all cases we matched those coefficients 
with the numerically computed values of~\cite{GZ1}. In~\cite{DG}, it 
was shown that the constant term of 
$\phi_{\gamma,1}(\hbar)$ is a topological invariant, but to date the 
full topological invariance of $\phi_{\gamma,1}(\hbar)$ is unknown.

Finally, we ought to point out a close connection between our series
$\phi_{\gamma,\zeta}(\hbar)$ and 
\begin{itemize}
\item
The radial asymptotics of Nahm sums at
complex roots of unity. This connection was observed at $\zeta=1$ during
conversations of the second author and Zagier in Bonn in the spring of 2012.
\item
The evaluations of one-dimensional state-integrals at rational
points \cite{GK1,GK2}. 
\item
The formula for an algebraic unit attached to an element of the Bloch 
group of a number field and a complex root of unity, 
appearing in \eqref{eq.phi0factor} \cite{calegari}.
\end{itemize}
These connections are not a coincidence; rather, they close a circle of ideas
motivated by several years of work on asymptotics of 
hypergeometric sums, quantum invariants, and their geometry and physics.


\section{The definition of $\phi_{\gamma,\zeta}$}
\label{sec.def}

\subsection{Ideal triangulations and Neumann-Zagier data}
\label{sub.NZ}

Ideal triangulations were introduced by Thurston as an efficient way to
describe (algebraically, or numerically) 3-dimensional hyperbolic
manifolds. For a leisure introduction, the reader may consult Thurston's
original notes~\cite{Th}, the exposition of Neumann--Zagier~\cite{NZ}
and Weeks~\cite{Weeks} and the documentation of {\tt SnapPy}~\cite{snappy}.
The shape of a 3-dimensional hyperbolic tetrahedra is a complex number
$z \in\BC\setminus\{0,1\}$. Letting $z' = (1-z)^{-1}$ and $z''=1-z^{-1}$,
the edges of an oriented ideal tetrahedron 
of shape $z$ can be assigned complex numbers 
according to Figure~\ref{fig.tet}.

\begin{figure}[htb]
\includegraphics[width=2in]{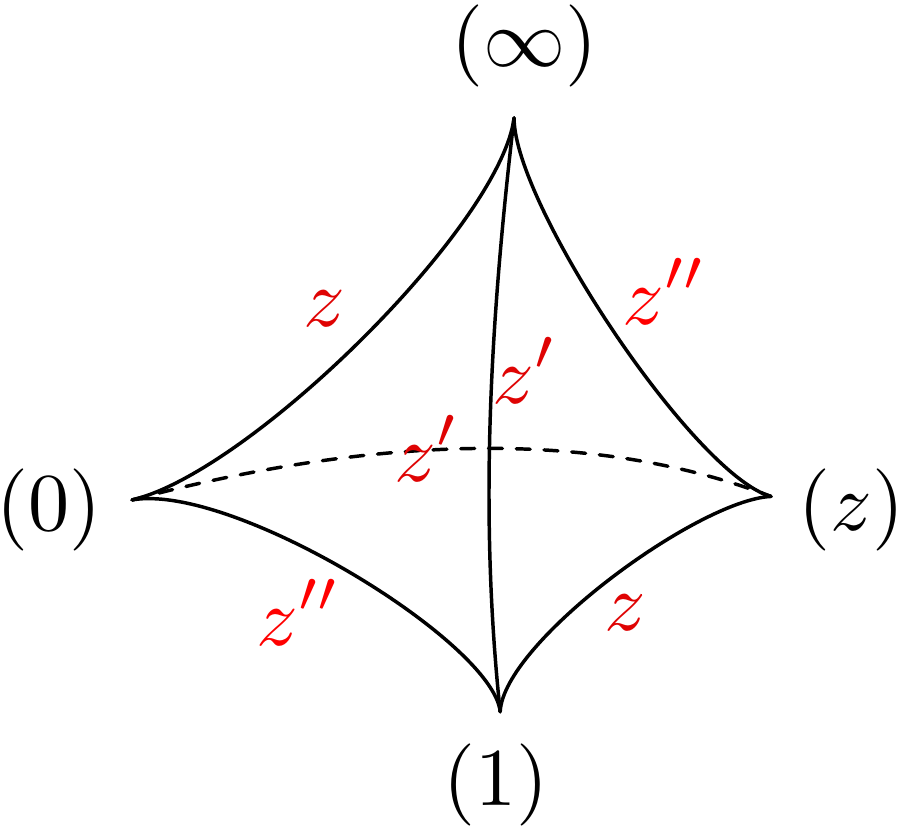}
\caption{An ideal tetrahedron with a shape assignement.}
\label{fig.tet}
\end{figure}

Let $M$ be an oriented hyperbolic manifold with one cusp (for instance
a hyperbolic knot complement) and $\calT$ an ideal triangulation of $M$ 
containing $N$ tetrahedra. In~\cite{DG} the authors introduced a 
Neumann-Zagier datum of $\calT$. The latter is a tuple
$\ga=(\mb A, \mb B,\nu,z,f,f'')$ that consists of:
\begin{itemize}
\item[(a)] Two matrices $\mb A, \mb B\in \mathrm{GL}(N,\Z)$ and a vector 
$\nu\in \Z^N$ encoding the coefficients of Thurston's gluing equations for 
the triangulation ($N-1$ independent equations imposing trivial holonomy 
around edges, and one equation imposing parabolic holonomy around the cusp).
\item[(b)] An $N$-tuple $z=(z_1,...,z_N)\in \C\backslash\{0,1\}$ of shape 
parameters, with each $z_i$ parametrizing the shape of the $i$-th 
tetrahedron, satisfying the gluing equations in the form 
$z^{\mb A}z''{}^{\mb B} = (-1)^\nu$, \ie\
\be 
\prod_i z_i^{A_{ji}}(1-z_i^{-1})^{B_{ji}} = (-1)^{\nu_j} \quad 
\text{for all} \quad j=1,\ldots, N\,.
\ee
\item[(c)] Two $N$-tuples $f,f''\in \Z^N$ satisfying
\be \mb A f + \mb B f'' = \nu\,. \ee
These provide a combinatorial flattening in the sense of 
\cite{Neumann-combi}. The integers $f,f''$, and $f'=1-f-f''$ also label 
edges of tetrahedra, with the property that the sum around any edge of the 
triangulation is $2$.
\end{itemize}
%
The Neumann-Zagier datum depends not just on the triangulation $\calT$ but 
also on which edges of each tetrahedron are labelled by the distinguished 
shape parameter $z_i$; this $3^N$-fold choice has been called a choice of 
``quad'' or ``gauge.''.

Neumann and Zagier~\cite{NZ} proved that $(\mb A\;\mb B)$ forms the top 
half of a symplectic matrix, \ie\ that $\mb A\mb B^T$ is symmetric and
$(\mb A\;\mb B)$ has full rank. It follows that if $\mb B$ is invertible, 
then $\mb B^{-1}\mb A$ is symmetric. 
We will call a Neumann-Zagier datum $\BZ$-\emph{nondegeretate} if $\mb B$ is 
invertible over the integers.

Fix a positive integer $k$. If $\ga$
is a Neumann-Zagier datum, let
\be 
\label{eq.zeta.theta}
\zeta = e^{\frac{2\pi i}{k}}\,,\qquad \theta_i = z_i^{1/k}\,,
\ee
where $\theta_i$ are chosen so that $\theta_i^k=z_i$. 
This defines number fields 
$F = \BQ(z_1,...,z_N)$, $F_k = F(\zeta)$, and $F_{G,k} 
=F_k(\theta_1,...,\theta_N)$, such that
\be 
F\subset F_k\subset F_{G,k}\,.
\ee
Observe that $F_{G,k}/F_k$ is the abelian Galois (Kummer) extension 
with group $G=(\BZ/k\BZ)^N=\la \sigma_1,\dots, \sigma_N \ra$ where 
\be
\label{eq.sigma}
\sigma_j(\th_i)=\zeta^{-\delta_{i,j}} \th_i 
\ee
where $\delta_{i,j}$ is Kronecker's delta function.
For the basic properties of Kummer theory, see~\cite[Sec.VI.8]{Lang}.

Below, we will construct a series $\phi_{\gamma,\zeta}(\hbar)$ 
for the $k$-th root of unity $\zeta = \exp(2\pi i/k)$. Then, after proving 
in Theorems \ref{thm.1loop} and  \ref{thm.1} that the coefficients in this 
series belong to $F_k$, the series $\phi_{\gamma,\zeta^p}(\hbar)$ for any other 
$k$-th root of unity $\zeta^p = \exp(2\pi ip/k)$ can be obtained from 
$\phi_{\gamma,\zeta}(\hbar)$ by a Galois automorphism.

\subsection{The $1$-loop invariant at level $k$}

Fix a $\BZ$-non-degenerate Neumann-Zagier datum $\gamma$ 
and a positive integer $k$. We use the notation of the 
previous section. For $m\in (\BZ/k\BZ)^N$, we define 
\be
\label{eq.amth}
a_m(\theta) = e^{-i\pi\, m\cdot\mb B^{-1}\mb A m} 
\zeta^{\frac12 \big[m\cdot \mb B^{-1}\mb A m+ m\cdot \mb B^{-1}\nu\big]}
\prod_{i=1}^N
\frac{\th_i^{-(\mb B^{-1} \mb A m)_i}}{
(\zeta \th_i^{-1};\zeta )_{m_i}}\,.
\ee
We also recall the {\em cyclic quantum dilogarithm} defined by
\be
D_k(x)=\prod_{s=1}^{k-1}(1-\zeta^{s}x)^s \qquad
D^*_k(x)=\prod_{s=1}^{k-1}(1-\zeta^{-s}x)^s \,.
\ee
This function appears in~\cite[Eqn.C.3]{Kashaev:star} and 
\cite[Eqn.2.30]{Kashaev:quantum.hyper}. 

\begin{definition}
\label{def.1loop}
With the above assumptions, the level $k$ 1-loop invariant of $\ga$ is
\be
\label{eq.tau}
\tau_{\gamma,k} :=  \frac{1}{k^{N/2}\sqrt{\det(\mb A\Delta_z''
+\mb B\Delta_z^{-1}) z^{f''/k}z''^{-f/k}}} 
\prod_{i=1}^N D_k^*(\th_i^{-1})^{1/k} \sum_{m\in (\Z/k\Z)^N} a_m(\th)\,,
\ee
where $\Delta_z = \diag(z_1,...,z_N)$ and $\Delta_{z''}=\diag(z_1'',...,z_N'')$ 
are diagonal matrices.
\end{definition}
Note that $\tau_{\gamma,k}$ depends on the Neumann-Zagier datum $\ga$,
the $k$-th root of unity $\zeta$ but also on the choice of $k$-th roots 
$\theta_i$ of $z_i$. The next theorem implies that 
$\tau_{\gamma,k}^{2k}$ depends only on $\ga$ and $\zeta$,
and therefore that $\tau_{\gamma,k}$ is well defined modulo multiplication 
by a $2k$-th root of unity.
The proof (given in Section \ref{sec.proofs}) follows from results of 
Zagier and the second author~\cite{GZ1} via a 
a comparison of an arithmetic to a geometric mean over the 
Galois group of $F_{G,k}/F_k$, reminiscent of Hilbert's theorem 90. 
\begin{theorem}
\label{thm.1loop}
We have $\tau_{\gamma,k}^{2k} \in F_k$.
\end{theorem}

\begin{remark}
\label{rem.Sunit}
It is easy to see that $\tau_k^{2k}/\tau_1^{2k}$ is an $S$-unit of the ring
of integers of $F_k$ where $S=\la z,1-z\ra \subset F_K^*$. For an
illustration, see Section~\ref{sec.data}.
\end{remark}

\begin{remark}
\label{rem.tau.alternative}
After replacing $\zeta$ by $\zeta^{-1}$,
We can give an alternative formula
for the $1$-loop invariant at level $k$ as follows:
\be
\label{eq.tau.alternative}
\tau_{\gamma,k} :=  \frac{1}{k^{N/2}\sqrt{\det(\mb - A\Delta_z''
-\mb B\Delta_z^{-1}) z^{f''/k}z''^{-f/k}}}
\prod_{i=1}^N \frac{z_i^{\frac{k-1}{2k}} (z_i'')^{\frac{k-1}{k}}}{
D_k(\zeta^{-1}\th_i)^{1/k}} 
\sum_{m\in (\Z/k\Z)^N} b_m(\th)\,,
\ee
where 
\be
\label{eq.alt.amth}
b_m(\theta) = e^{i\pi\, L \cdot m} 
\zeta^{-\frac12 \big[m^T Q m+ L \cdot m\big]}
\prod_{i=1}^N
\frac{\th_i^{ Q_i \cdot m}}{
(\zeta^{-1} \th_i;\zeta^{-1} )_{m_i}}\,
\ee
and
\be
\label{eq.QL}
L= -\mb B^{-1} \nu + (1,\dots,1)^T, \qquad Q=I-\mb B^{-1}\mb A \,.
\ee
\end{remark}

\subsection{The $n$-loop invariants  at level $k$ for $n \geq 2$}
\label{sub.all}

The definition of the higher-loop invariants $S_{\gamma,n,k}$ is motivated 
by perturbation theory of the state-integral model for complex 
Chern-Simons theory, reviewed briefly in Section \ref{sec.CCS}. 
In this section we define the higher-loop invariants using 
formal Gaussian integration, and in the next section we give a 
Feynman diagram formulation of the higher-loop invariants.

Fix a $\BZ$-non-degenerate Neumann-Zagier datum $\ga$ and a 
positive integer $k$. We will use
the notation of the previous section. 
If $f: (\BZ/k\BZ)^N \longto \BC$, we define
\be
\label{eq.Avf}
\Av(f)= \frac{\sum_{m\in (\Z/k\Z)^N} a_m(\th) f(m)}{
\sum_{m\in (\Z/k\Z)^N} a_m(\th)}\,,
\ee
assuming that the denominator is nonzero.
Consider the symmetric matrix
\be
\CH=\frac1k(-\mb B^{-1}\mb A+\Delta_{z'})\,,
\ee
where $\Delta_{z'}=\diag(z_1',...,z_N')$. Assuming that $\CH$ is invertible,
a formal power series $f_\hbar(x)\in \BQ(z)[\![x,\hbar^{\frac12}]\!]$ has a 
{\em formal Gaussian integration}, given by 
\be 
\label{formalG}
\la f_{\hb}(x)\ra = \frac{\int dx \, 
e^{-\frac12x^T \CH \,x}f_{\hbar}(x)}{\int dx \, e^{-\frac12x^T \CH \,x}}\,.
\ee
This integration,
which is a standard tool of perturbation theory in
physics, and may be found in numerous texts (e.g. \cite{BIZ})
is defined by expanding $f_\hbar(x)$ as a series in $x$, 
and then formally integrating each monomial, using the quadratic form 
$\CH^{-1}$ to contract $x$-indices pairwise. 

The building block of each tetrahedron is the power series
\be
\label{eq.block}
\psi_\hbar(x,\th,m)=\exp\left(
\sum_{n=1}^\infty\sum_{j=0}^\infty \frac{\hbar^{n-1}(-1)^j}{
n!j!k^j} \sum_{s=1}^k B_n\left(\frac{s}{k}\right) \Li_{2-n-j}(\zeta^{m+s}
\th^{-1}) x^j
\right)
\ee
For an ideal triangulation $\calT$ with $N$ tetrahedra, a natural number $k$, 
and $m \in (\Z/k\Z)^N$, we define 
%
\be
\label{eq.fT}
f_{\calT,\hbar}(x;\th,m)=\exp \bigg(-\frac{\hbar^{\frac12}}{2k} x^T \mb B^{-1} 
\nu +\frac\hbar{8k} f^T\mb B^{-1}\nu \bigg)
\prod_{i=1}^N \psi_\hbar(x_i,\th_i,m_i)
\ee

\begin{definition}
\label{def.phih}
We define
\be
\label{eq.defphi}
\phi^+_{\ga,\zeta}(\hbar)=\Av(\la f_{\calT,\hbar}(x;\th,m) \ra) \in 
1+\hbar\BC[\![\hbar]\!] \,.
\ee
\end{definition}
\noindent 
Thus, we can write
\be
\label{def.Sn}
\phi^+_{\ga,\zeta}(\hbar) = 
\exp\left(\sum_{n=1}^\infty S_{\ga,n+1,k} \hbar^n\right) \,.
\ee
We call $S_{\ga,n,k}$ the level $k$, $n$-loop invariant of $\ga$.
We finally define
\be 
\phi_{\gamma,\zeta}(\hbar) = \tau_{\gamma,\zeta}
\phi^+_{\gamma,\zeta}(\hbar)\,. 
\ee

\begin{theorem}
\label{thm.1}
The coefficients of the power series $\phi^+_{\ga,\zeta}(\hbar)$ are in
$F_k$.
\end{theorem}
In particular, the series $\phi^+_{\ga,\zeta}(\hbar)$ depends on $\gamma$,
the $k$-th root of unity $\zeta$, but it is independent of the
choice of $k$-th roots $\th_i$ of $z_i$. 
The above theorem is not trivial since $a_m(\th)$ is an element of the
larger field $F_{G,k}$, whereas the coefficients of the above average are
claimed to be in the field $F_k$. 
For the proof, see Section \ref{sec.proofs}.

\begin{remark}
\label{rem.primitivek}
Theorems~\ref{thm.1loop} and \ref{thm.1} remain valid if $\zeta$ denotes 
a fixed primitive $k$-th root of unity instead of $\zeta=e^{2 \pi i/k}$.
Probably a better notation is $S_{\ga,n,\zeta}$ rather than $S_{\ga,n,k}$
which is valid for all primitive $k$-roots of unity $\zeta$.
\end{remark}

\subsection{Feynman diagrams for the $n$-loop invariant}
\label{sub.diagrams}

In this section we give a Feynman diagram formulation of the higher-loop
invariants.
A Feynman diagram $\mb D$ is a finite graph possibly with loops and multiple
edges. To every edge in a Feynman diagram we associate the symmetric 
$N\times N$ propagator matrix
\be  \Pi =  \hbar k (-\mb B^{-1}\mb A + \Delta_{z'})^{-1}\,, \ee
and to a vertex with valence $j$ we associate the vertex factor 
$\Gamma^{(j)}$, which is a tensor of rank $j$ whose only nonzero entries 
$\Gamma_i^{(j)} := \Gamma_{ii...i}^{(j)} \in F_{G,k}(\!(\hbar)\!)$ lie on the 
diagonal, and are functions of $m\in (\BZ/k\BZ)^N$ and $\theta$,
\be
\begin{array}{rl}
\Gamma^{(0)} &= \displaystyle \frac{\hbar}{8k}f\mb B^{-1}\nu 
+ \sum_{n=2}^\infty \frac{\hbar^{n-1}}{n!}\sum_{s=1}^k B_n
\big(\tfrac sk\big) \sum_{i=1}^N\Li_{2-n}(\zeta^{m_i+s} \theta_i^{-1}) \\[.1cm]
\Gamma_i^{(1)} &=  \displaystyle -\frac{1}{2k}(\mb B^{-1}\nu)_i 
- \sum_{n=1}^\infty \frac{\hbar^{n-1}}{k\,n!}\sum_{s=1}^k B_n
\big(\tfrac sk\big)\Li_{1-n}(\zeta^{m_i+s}\theta_i^{-1}) \\[.1cm]
\Gamma_i^{(2)} &= \displaystyle \sum_{n=1}^\infty 
\frac{\hbar^{n-1}}{k^2\,n!}\sum_{s=1}^k B_n\big(\tfrac sk\big)
\Li_{-n}(\zeta^{m_i+s}\theta_i^{-1}) \\[.1cm]
\Gamma_i^{(j)} &= \displaystyle \sum_{n=0}^\infty 
\frac{(-1)^j\hbar^{n-1}}{k^j\,n!}\sum_{s=1}^k B_n\big(\tfrac sk\big)
\Li_{2-n-j}(\zeta^{m_i+s}\theta_i^{-1})\qquad (j\geq 3)\,,
\end{array}
\ee
%
where $B_n(x)$ are the Bernoulli polynomials, defined by 
$te^{xt}/(e^t-1) = \sum_{n\geq 0}B_n(x)t^n/n!$ and $F_{G,k}(\!(\hbar)\!)$ denotes
the ring of formal Leurent series in $\hbar$ with coefficients in $F_{G,k}$.

Note that each $\Gamma_i^{(j)}$ only depends on $m,\,\theta$ through the 
combination $\zeta^{m_i}\theta_i^{-1}$. Moreover, all the $l$-polylogarithms 
appearing here involve non-positive $l$, hence are rational functions. 
The evaluation $[\mb D]_m$ of a diagram is obtained by contracting 
propagator and vertex indices, and multiplying by a standard symmetry 
factor $1/|\sigma(\mb D)|$, where $\sigma(\mb D)$ is the diagram's symmetry 
group,
\be [\mb D]_m =  \frac{1}{|\sigma(\mb D)|} 
\sum_{\text{coincident indices}} \;\;\prod_{\text{edges $e$}} 
\Pi(e)\prod_{\text{vertices $v$}}  \Gamma(v)\,. \ee
For example, the diagram in the center of the top row of Figure 
\ref{fig.diags2intro} has an evaluation $[\mb D]_m = 
\frac18 \sum_{i,i'=1}^N \Pi_{ii}\Gamma^{(3)}_i \Pi_{ii'}\Gamma^{(3)}_{i'}\Pi_{i'i'}$. 
To the \emph{trivial} diagram $\bullet$ that consists of one vertex and
no edges we associate the vacuum energy 
$[\bullet]_m=\Gamma^{(0)}$. The next Lemma follows from evaluating the 
formal Gaussian integral~\eqref{eq.defphi} in terms of Feynman diagrams;
see~\cite{BIZ}.

\begin{lemma}
\label{lem.phih}
For a $\BZ$-non-degenerate Neumann-Zagier datum $\gamma$, we have
\be 
\phi^+_{\gamma,\zeta}(\hbar)  = 
\Av \bigg[ \exp\bigg( \sum_{\text{connected $\mb D$}} [\mb D]_m\bigg)\bigg]\, 
\in 1 + \hbar\,\BC[\![\hbar]\!]\,, 
\ee
where the sum is over all connected diagrams $\mb D$, including the empty 
diagram. 
\end{lemma}
Using the above Lemma and Equation \eqref{def.Sn}, it follows 
that in order to compute $S_{\gamma,n}$ for $n\geq 2$, it 
suffices to consider the finite set of Feynman diagrams with
\be 
\label{eq.LD}
\text{\#(1-vertices) + \#(2-vertices) + \#(loops)} \leq n \,,
\ee
and to truncate the formal power series in each of the vertex factors to 
finite order in $\hbar$. In the next two sections, we give axplicit
formulas for the 2 and 3-loop invariants.

\subsection{The $2$-loop invariant in detail}
\label{sub.2loop}

The six diagrams that contribute to $S_{\ga,2}$ are shown in Figure 
\ref{fig.diags2intro}, together with their symmetry factors.
Their evaluation gives the following formula for $S_{\ga,2,k}$:
\begin{align} \label{2loopexplicit}
S_{\ga,2,k} &= \Av\left( \coeff\left[\Gamma^{(0)}+ 
\frac18  \Gamma^{(4)}_i (\Pi_{ii})^2 
+ \frac18\Pi_{ii}\Gamma^{(3)}_i\Pi_{ij}\Gamma^{(3)}_j\Pi_{jj} 
+ \frac{1}{12}\Gamma^{(3)}_i(\Pi_{ij})^3\Gamma^{(3)}_j \right. \right.\\
&\hspace{1in}
\left.\left. + \frac12 \Gamma^{(1)}_i\Pi_{ij}\Gamma^{(3)}_j\Pi_{jj} 
+ \frac12 \Gamma^{(2)}_i\Pi_{ii} 
+ \frac12\Gamma^{(1)}_i\Pi_{ij}\Gamma^{(1)}_j ,\; \hbar\right]
\right) \,, 
\notag
\end{align}
%
where the dependence of vertex factors on $m$ is suppressed;
all the indices $i$ and $j$ are implicitly summed from $1$ to $N$;
and $\coeff[f(\hb),\,\hb^k]$ denotes the coefficient of $\hb^k$ 
in a power series $f(\hb)$.

\begin{figure}[htb]
\centering
\includegraphics[width=4in]{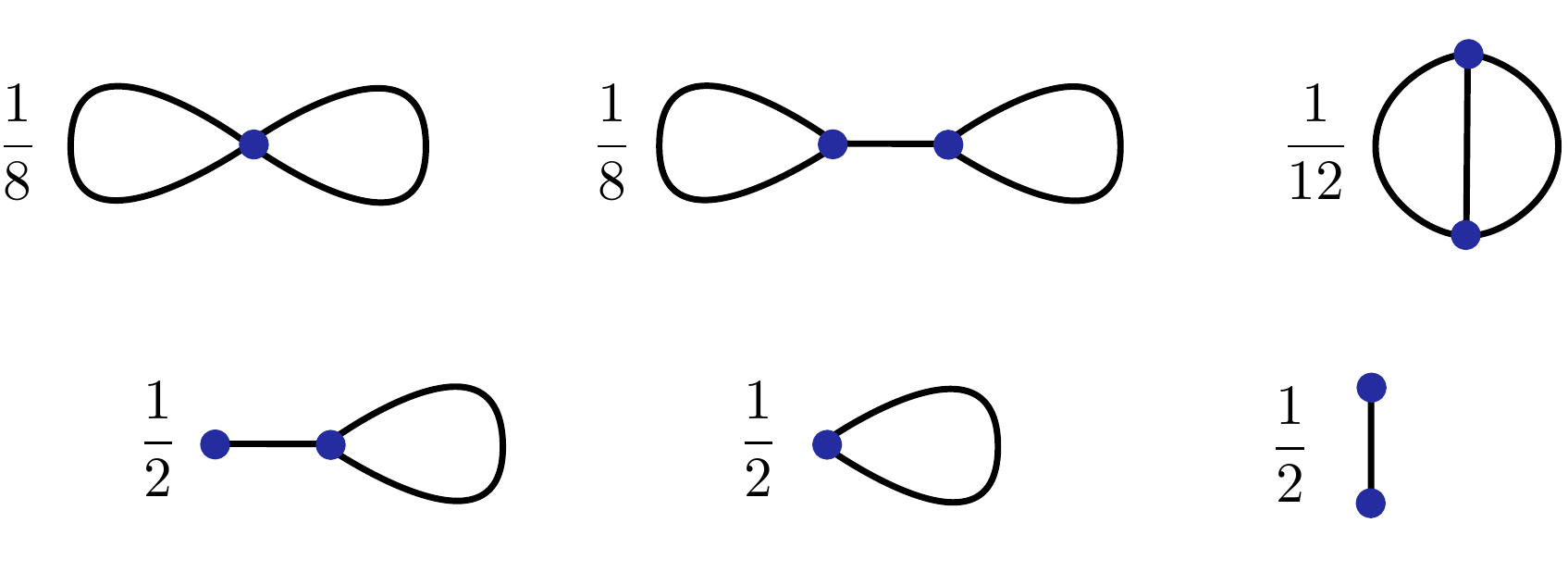}
\caption{Diagrams contributing to $S_{\ga,2}$ with symmetry factors. The top 
row of diagrams have exactly two loops, while the bottom row have fewer 
loops and additional $1$-vertices and $2$-vertices.}
\label{fig.diags2intro}
\end{figure}

Concretely, the 2-loop contribution from the vacuum energy is 
$$
\Gamma^{(0)}=\frac{1}{8k}f \mb B^{-1}\nu-\frac12
\sum_{s=1}^k(\frac{s^2}{k^2}-\frac{s}{k}+\frac16)
\sum_{i=1}^N(1-\zeta^{-m_i-s}\theta_i)^{-1} \,.
$$
The four other vertices contribute only at leading order; abbreviating 
$\tilde\theta_i=\zeta^{-m_i-s}\theta_i$ and 
$\tilde \theta_i' = (1-\zeta^{-m_i-s}\theta_i)^{-1}$, they are
\be
\notag  \Gamma^{(1)}_i=-\frac{1}{2k}(\mb B^{-1}\nu)_i + 
\frac1k\sum_{s=1}^k\Big(\frac sk-\frac12\Big)\tilde\theta_i'\,, 
\ee
\be
\notag \Gamma^{(2)}_i = \frac{1}{k^2}\sum_{s=1}^k \Big(\frac sk-\frac12\Big) 
\tilde\theta_i\tilde\theta_i'{}^2\,,
\qquad
\Gamma^{(3)}_i = -\frac{\tilde\theta_i\tilde\theta_i'^2}{k\hbar}\,,\qquad 
\Gamma^{(4)}_i= 
-\frac{\tilde\theta_i(1+\tilde\theta_i)\tilde\theta_i'{}^3}{k\hbar}\,.
\ee

\begin{figure}[htpb]
\centering
\includegraphics[width=6.5in]{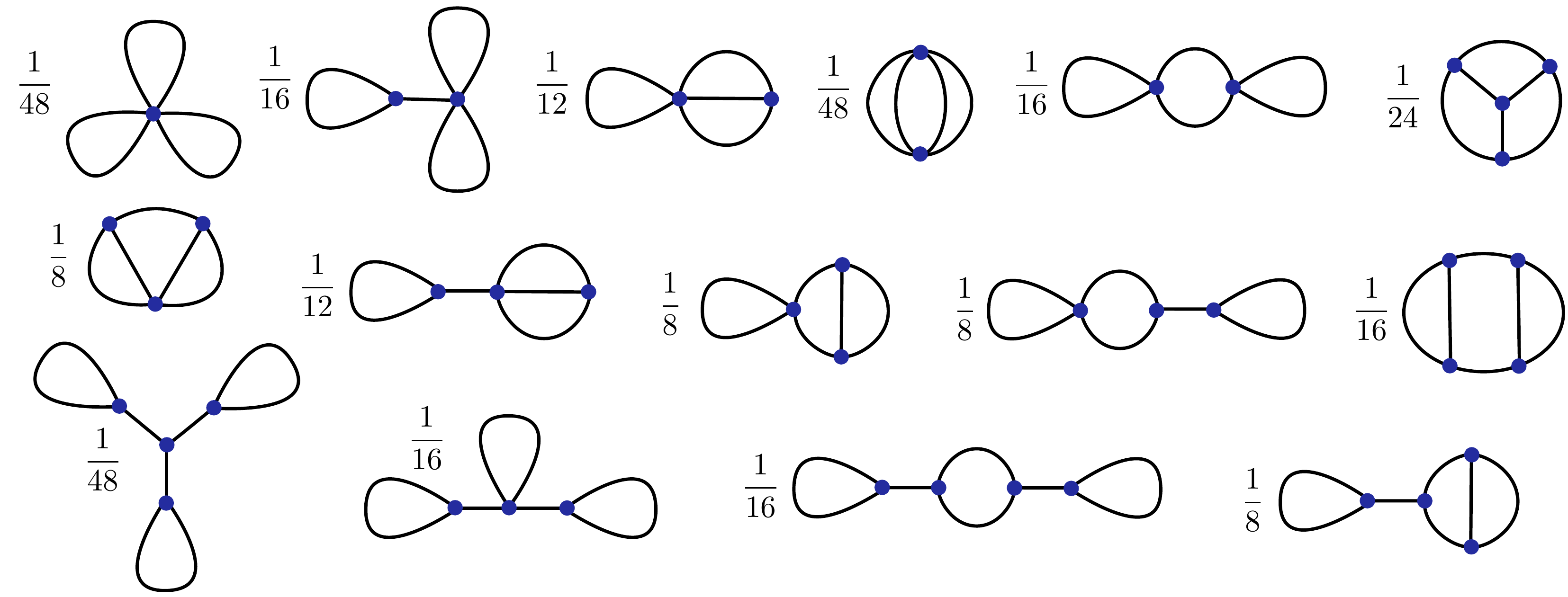}
\caption{Diagrams with three loops contributing to $S_3$.}
\label{fig.diags3}
\end{figure}

\begin{figure}[htpb]
\centering
\includegraphics[width=6.5in]{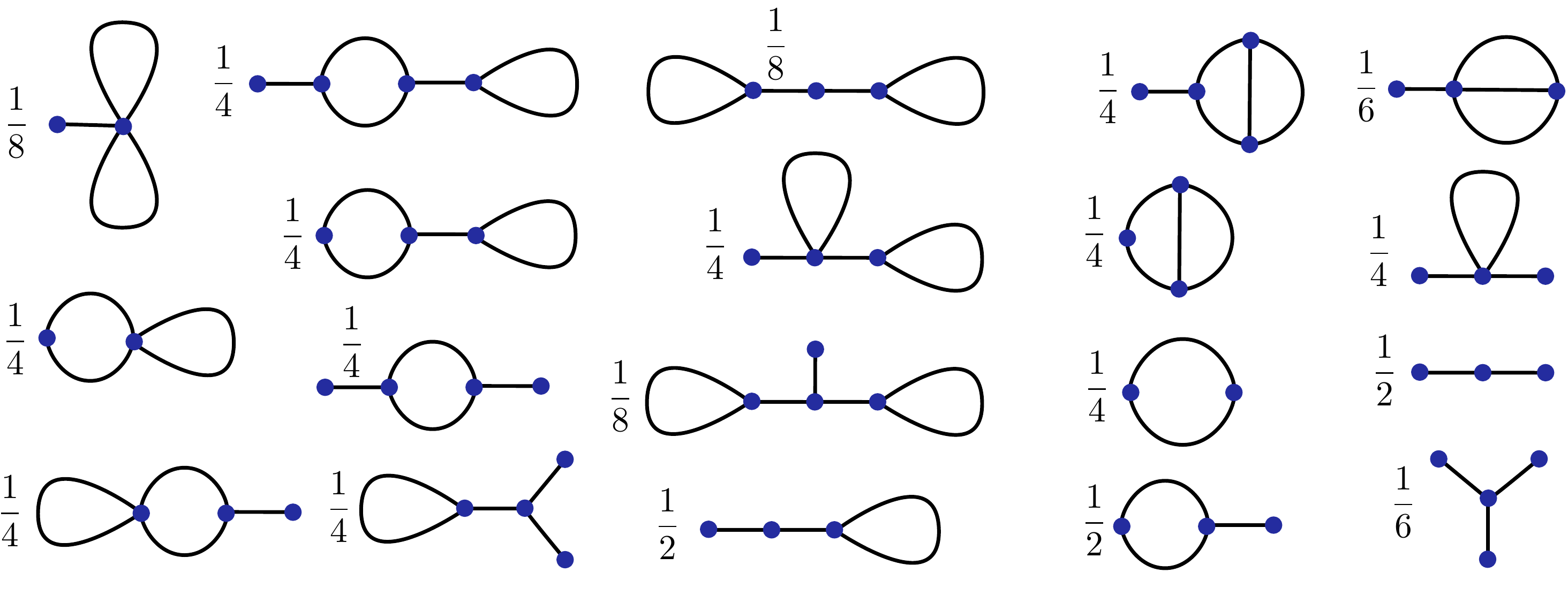}
\caption{Diagrams with $1$-vertices and $2$-vertices 
contributing to $S_3$.}
\label{fig.diags3dot}
\end{figure}

\subsection{The $3$-loop invariant}
\label{sub.3loop}

For the next invariant $S_{\ga,3,k}$, all the diagrams of Figure 
\ref{fig.diags2intro} contribute, collecting the coefficient
of $\hb^2$ of their evaluation. In addition, there are $34$ new diagrams 
that satisfy the inequality \eqref{eq.LD}; they are shown in Figures 
\ref{fig.diags3} and \ref{fig.diags3dot}. Calculations indicate that 
the 3-loop invariant $S_{\ga,3,k}$ is well defined, and invariant under 2-3
moves. The invariants $\tau_{\ga,k}, S_{\ga,2,k}, S_{\ga,3,k}$ have been 
programmed in {\tt Mathematica} as well as in {\tt python} and take as input a 
Neumann-Zagier datum readily available from {\tt SnapPy} \cite{snappy}. 

The number of diagrams that contribute to the $n$-loop invariant 
is given in Table \ref{tab.diagrams}.
For large $n$, we expect that $n!^2 C^n$ diagrams contribute to 
the $n$-loop invariant. It would be nice to find a more efficient computation.

\begin{table}[htb]
\begin{center}
\begin{tabular}{|l|l|l|l|l|l|}\hline
 $n$ & $2$ & $3$ & $4$ & $5$ & $6$ \\ \hline
$g_n$ & $6$ & $40$ & $331$ & $3700$ & $53758$ \\ \hline
\end{tabular}\vspace{.2cm}
\caption{The number $g_n$ of graphs that contribute to the $n$-loop 
invariant for $n=2,\dots,6$.} 
\label{tab.diagrams}
\end{center}
\end{table}

\subsection{Matching with the numerical asymptotics of the Kashaev
invariant}
\label{sub.matching}


Numerical asymptotics of the Kashaev invariant were obtained by Zagier
and the second author in~\cite{GZ1} for several knots, summarized in Table 
\ref{tab.num.kashaev}.
Our $n$-loop invariants at level $k$, presented in Section~\ref{sec.compute},
agree with the numerical computations of~\cite{GZ1}. This is a strong 
consistency test for all computational methods. 

\begin{table}[htb]
\begin{center}
\begin{tabular}{|c|c|c|}\hline
Knot & Level & Loops \\ \hline
$4_1$ & $\leq 7$ & $\leq 5$ \\ \hline
$5_2$ & $\leq 5$ & $\leq 4$ \\ \hline
$(-2,3,7)$ & $\leq 5$ & $\leq 4$ \\ \hline
$(-2,3,-3)$ & $\leq 3$ & $\leq 5$ \\ \hline
$(-2,3,9)$ & $\leq 3$ & $\leq 5$ \\ \hline
$6_1$ & $\leq 3$ & $\leq 3$ \\ \hline
\end{tabular}\vspace{.2cm}
\caption{Numerical asymptotics of the Kashaev invariant from~\cite{GZ1}.} 
\label{tab.num.kashaev}
\end{center}
\end{table}

\subsection{Topological invariance}
\label{sub.inv}

We conjectured in the introduction that $\phi_{\gamma,\zeta}(\hbar)$ is 
actually a topological invariant --- depending only on a knot $K$ and the 
root of unity $\zeta$, rather than on a full Neumann-Zagier datum --- and 
thus $\phi_{\gamma,\zeta}(\hbar) = \phi_{K,\zeta}(\hbar)$ is the series that 
appears in the right hand side of the Quantum Modularity Conjecture. 
We can now make this a bit more precise.

We begin with the following experimental observations:
\begin{itemize}
\item 
For all 502 hyperbolic knots with at most 8 ideal tetrahedra in the 
{\tt CensusKnots}, their default {\tt SnapPy} triangulations are 
$\BZ$-nondegenerate, in the sense that there is a gauge 
(for a definition, see Section~\ref{sub.NZ}) for which $|\!\det\, \mb B|=1$.
\item 
The default and the canonical {\tt SnapPy} triangulation of the $5_2$
and $6_1$ knots have several gauges for which $|\!\det\, \mb B|=1$.
For each of the above knot we have checked that $(\tau_{\gamma,k})^{12k}$
is independent of the above gauges, and that $S_{\gamma,2,k}$ 
is independent up to addition of $\BZ/{24k}$, and that $S_{\gamma,3,k}$ is  
independent. These are slightly better than the general ambiguities that 
appear in the nonperturbative level-$k$ state-integral (see Section 
\ref{sec.CCS}), which motivated the definitions above. Moreover, 
(local) triangulation-invariance of the  
level-$k$ state integral suggests that $S_{\gamma,n,k}$ should be independent 
of triangulation and $\gamma$ for all $n$. 
\end{itemize}
We are therefore led to conjecture
\begin{conjecture}
\label{conj.inv}
For any knot $K$, there exists a triangulation with a non-degenerate 
Neumann-Zagier datum $\gamma$. The series $\phi_{\gamma,\zeta}(\hbar)$ is 
independent of the choice of triangulation and $\gamma$, up to multiplication 
by $\zeta^{\frac1{12}}$ and $e^{\hbar/24k}$. Modulo these ambiguities, 
$\phi_{\gamma,\zeta}(\hbar)=\phi_{K,\zeta}(\hbar)$ equals the series on the 
right hand side of the Quantum Modularity Conjecture.
\end{conjecture}

Note that if $F_k$ does not contain a primitive third root of unity (for 
example, if $F$ is sufficiently generic and $3\nmid k$) then topological 
invariance of $(\tau_{k})^{12 k}$ together with Theorem \ref{thm.1loop} 
implies topological invariance of $(\tau_k)^{4k}$.

%
%
%
%




\section{Proofs}
\label{sec.proofs}

\subsection{Proof of Theorem~\ref{thm.1}}
\label{sub.proof.thm1}

First, we need to prove that the summation in Equation~\eqref{eq.tau} is
well-defined over the set $(\BZ/k\BZ)^N$. This was observed in~\cite{GZ1}
and uses the fact that $\th^k$ is a solution to the Neumann-Zagier equations.

\begin{lemma}
\label{lem.periodic}\cite{GZ1}
The summation in Equation~\eqref{eq.tau} is $k$-periodic.
\end{lemma}

\begin{proof}
We need to show that
$$
a_{m+k e_j}(\theta)=a_m(\theta) \,.
$$
This follows from the definition of $a_m(\th)$ given in 
Equation~\eqref{eq.amth}, the fact that $z$ is a solution to
the Neumann-Zagier equations, and Lemma~\ref{lem.cyclic} for the cyclic 
dilogarithm.
\end{proof}

To prove Theorem~\ref{thm.1}, recall the Galois extension $F_{G,k}/F_k$, 
and the element $a_m(\th)$ of $F_{G,k}$. Let $\sigma_j$ denote the 
$j$-th generator of the Galois group from Equation~\eqref{eq.sigma}, and let 
\be
\e_j(\th)=a_{e_j}(\th)^{-1} \,.
\ee
The next lemma was observed in~\cite{GZ1}.

\begin{lemma}
\label{lem.asigma}\cite{GZ1}
For all $m$ we have: 
\be
\frac{a_{m}(\sigma_j\th)}{a_{m+e_j}(\th)} = \e_j(\th) \,.
\ee
\end{lemma}

\begin{proof}
It suffices to show that the left hand side of the above equation is
independent of $m$, since the right hand side is the value at $m=0$.
To prove this claim, we compute
\begin{align*}
\frac{a_{m}(\sigma_j\th)}{a_{m+e_j}(\th)} &=
\frac{e^{\pm \pi i \mb B^{-1}\nu \cdot m}}{
e^{\pm \pi i \mb B^{-1}\nu \cdot (m+e_j)}} \cdot
\frac{\zeta^{\frac12 \left(m^T\mb B^{-1}\mb A m \pm \mb B^{-1}\nu \cdot m\right)}}{
\zeta^{\frac12 \left((m+e_j)^T\mb B^{-1}\mb A (m+e_j) \pm \mb B^{-1}\nu \cdot (m+e_j)\right)}} 
\cdot \\ 
& \quad \prod_{i=1}^N \frac{(\sigma_j \th_i)^{-(\mb B^{-1} \mb A m)_i}}{
\th_i^{-(\mb B^{-1} \mb A (m+e_j))_i}} \cdot \prod_{i=1}^N
\frac{(\zeta \th_i^{-1};\zeta )_{m_i+\delta_{i,j}}}{
(\zeta \sigma_j \th_i^{-1};\zeta )_{m_i}}\,.
\end{align*}
We focus on the $m$-dependent part of each of the 4 fractions.
Obviously, 
$$
\frac{e^{\pm \pi i \mb B^{-1}\nu \cdot m}}{
e^{\pm \pi i \mb B^{-1}\nu \cdot (m+e_j)}} =
\left(\text{term independent of $m$}\right) \,.
$$
Next,
$$
\frac{\zeta^{\frac12 \left(m^T\mb B^{-1}\mb A m \pm \mb B^{-1}\nu 
\cdot m\right)}}{
\zeta^{\frac12 \left((m+e_j)^T\mb B^{-1}\mb A (m+e_j) \pm \mb B^{-1}\nu 
\cdot (m+e_j)\right)}} = \zeta^{-(\mb B^{-1}\mb A m)_j} \times 
\left(\text{term independent of $m$}\right).
$$
Splitting the product of the third fraction to the case when $j \neq i$
and the case when $j=i$ implies that
$$
\prod_{i=1}^N \frac{(\sigma_j \th_i)^{-(\mb B^{-1} \mb A m)_i}}{
\th_i^{-(\mb B^{-1} \mb A (m+e_j))_i}} 
= \zeta^{(\mb B^{-1}\mb A m)_j} \times 
\left(\text{term independent of $m$}\right) \,.
$$
Using the identity 
$$
\frac{(\zeta \th_j^{-1};\zeta)_{m_j+1}}{(\zeta^2 \th_j^{-1};\zeta)_{m_j}}=
1-\zeta \th_j^{-1} \,,
$$
and splitting the product of the fourth fraction to the case when $j \neq i$
and the case when $j=i$ implies that
$$
\frac{(\zeta \th_i^{-1};\zeta )_{m_i+\delta_{i,j}}}{
(\zeta \sigma_j \th_i^{-1};\zeta )_{m_i}} =
\left(\text{term independent of $m$}\right) \,.
$$
This completes the proof of the lemma.
\end{proof}

Using the fact that
the sum is $m$-periodic it follows that
\be
\label{eq.fs1}
\sigma_j \left( \sum_{m \in (\BZ/k\BZ)^N} a_m(\th) \right)= e_j(\th)
\sum_{m \in (\BZ/k\BZ)^N} a_m(\th)\,.
\ee
Suppose that
$f(m,\th)$ is a rational function of $\th$ with the property that 
\be
\label{eq.fksigma}
f(m,\sigma_j\th)=f(m+e_j,\th)
\ee
for all $j=1,\dots,N$ and all $m$. Then, it follows that for all $m$ we have
$$
a_m(\sigma_j \th) f(m,\sigma_j\th) = \e_j(\th) a_{m+e_j}(\th) f(m+e_j,\th) \,.
$$
Summing up, we obtain that 
\be
\label{eq.fs2}
\sigma_j \left( \sum_{m \in (\BZ/k\BZ)^N} a_m(\th) f(m,\th) \right)= \e_j(\th)
\sum_{m \in (\BZ/k\BZ)^N} a_m(\th) f(m,\th) \,.
\ee
Equations~\eqref{eq.fs1} and \eqref{eq.fs2} and the fact that $F_{G,k}/F_k$
is a Galois extension imply that if $f$ satisfies 
Equation~\eqref{eq.fksigma}, then 
$\Av(f(-,\theta))\in F_k$.

Now observe that the vertex weights of the Feynman diagrams are 
$\BQ(\zeta)$-linear combinations of values of the polylogarithm 
$\Li_l(\zeta^{m_i+s}\th_i^{-1})$ for $l <0$. Since the polylogarithm
is a rational function, and $\zeta^{m_i}\th_i^{-1}$ satisfies 
Equation~\eqref{eq.fksigma} for all $i$ and $j$, Theorem~\ref{thm.1} follows.
\qed

\subsection{Some identities of the cyclic dilogarithm}
\label{sub.cyclicdilog}

\begin{lemma}
\label{lem.cyclic}
We have:
\begin{subequations}
\begin{align}
\label{eq.cd1}
D_k(\zeta x) &= D_k(x) \frac{(1-x)^k}{1-x^k} \\
\label{eq.cd2}
D^*_k(\zeta^{-1} x) &= D^*_k(x) \frac{(1-x)^k}{1-x^k} \\
\label{eq.cd3}
D_k^*(x) & =\frac{(1-x^k)^{k-1}}{D_k(\zeta x)} \\
\label{eq.cd4}
D_k(x)   & = e^{\frac{2 \pi i}{3} (k^2-1)}
x^{\frac{k(k-1)}{2}} D^*_k(1/x)
\end{align}
\end{subequations}
\end{lemma}

\begin{proof} Parts (a) and (b) are straightforward.
For~\eqref{eq.cd3}, use
$$
D^*_k(x)=\prod_{s=1}^{k-1}(1-\zeta^{-s}x)^s
=\prod_{s=1}^{k-1}(1-\zeta^{s-k}x)^{k-s}
=\frac{\left(\prod_{s=1}^{k-1}(1-\zeta^{s}x)\right)^k}{
\prod_{s=1}^{k-1}(1-\zeta^{s}x)^{s}} = \left(\frac{1-x^k}{1-x}\right)^k 
\frac{1}{D_k(x)}
$$
and then apply~\eqref{eq.cd1}. For~\eqref{eq.cd4}, use
$$
D_k(x)=\prod_{s=1}^{k-1}(1-\zeta^{s}x)^s= \prod_{s=1}^{k-1} (-\zeta x)^{s}
(1- \zeta^{-s} x^{-1})^s = e^{\frac{2 \pi i}{3} (k^2-1)}
x^{\frac{k(k-1)}{2}} D^*_k(1/x)
\,.
$$
\end{proof}

\subsection{Proof of Theorem~\ref{thm.1loop}}
\label{sub.proof.1loop}

Let us define
\begin{align*}
S(\th) & =\sum_{m\in (\Z/k\Z)^N} a_m(\th) \\
P(\th) & = \prod_{i=1}^N D_k^*(\th_i^{-1})^{1/k} \,.
\end{align*}
To begin with, we have $S(\th), P^k(\th) \in F_{G,k}$.
If $\sigma_j$ is the $j$-th generator of the Galois group of $F_{G,k}/F_k$
then Equation~\eqref{eq.fs1} implies that 
$$
\sigma_j S(\th) = \e_j(\th) S(\th) \,.
$$
We claim that
\be
\label{eq.fpk}
\sigma_j P^k(\th) = \e_j(\th)^{-k} P^k(\th) \,.
\ee
Combined, they show that $S^k(\th) P^k(\th) \in F_k$ which implies 
Theorem~\ref{thm.1loop}. To prove Equation~\eqref{eq.fpk}, we 
separate the product when $i \neq j$ and when $i=j$ as follows:
\begin{align*}
\frac{\sigma_j P^k(\th)}{P^k(\th)} &=\left(\prod_{i=1,i\neq j}^N
\frac{D_k^*(\sigma_j \th_i^{-1})}{D_k^*(\th_i^{-1})} \right)\cdot
\frac{D_k^*(\sigma_j \th_j^{-1})}{D_k^*(\th_j^{-1})} \\ 
&= \frac{D_k^*(\zeta \th_j^{-1})}{D_k^*(\th_j^{-1})} \\
&= \frac{1-z^{-1}_j}{(1-\zeta \th_j^{-1})^k} \qquad 
\text{by Equation~\eqref{eq.cd2}} \,.
\end{align*}
On the other hand, $\th^k=z$ satisfies the Neumann-Zagier equations 
$$
z^{\mb A} z''^{\mb B} = (-1)^{\nu}
$$
Using the fact that $B$ is unimodular, we can write the above equations
in the form
$$
z''=(-1)^{\mb B^{-1} \nu} z^{-\mb B^{-1} \mb A} \,.
$$
In other words, for all $j=1,\dots,N$ we have
$$
1-z_j^{-1} = \prod_{i=1}^N z_i^{- (\mb B^{-1} \mb A e_j)_i} \,.
$$
Combining with the above, and the definition of $\epsilon_j(\th)$, 
concludes the proof of Equation~\eqref{eq.fpk}.
\qed


\section{Complex Chern-Simons theory}
\label{sec.CCS}

In this section we review in brief some of the physics of complex Chern-Simons 
theory, discuss the limits related to the Quantum Modularity Conjecture, and 
explain how to derive the definition of the series $\phi_{\gamma,\zeta}$ from 
Section~\ref{sec.def}.

\subsection{Basic structure}

Chern-Simons theory with complex gauge group $G_\C$ (where $G$ is a compact
Lie group) was initially studied by Witten in~\cite{Witten-CSgrav, Witten-cx}. 
It is a topological quantum field theory 
in three-dimensions, whose action is a sum of holomorphic and antiholomorphic 
copies of the usual Chern-Simons action
\be 
S(\calA, \ol\calA) = \frac{t}{8\pi} \,
I_{CS}(\calA) + \frac{\tilde t}{8\pi}\, I_{CS}(\ol\calA)\,,  
\ee
where $\calA$ is a connection on a $G_\C$ bundle over a 3-manifold $M$, 
and $I_{CS}(\calA) = \int_M (\calA d\calA+\frac23 \calA^3)$ (with additional 
boundary terms if $M$ is not closed). In order for the path-integral measure 
$\exp(-iS)$ to be invariant under all gauge transformations of $\calA$, the 
levels $t,\tilde t$ must obey the quantization condition
\be k:= \tfrac12(t + \tilde t)  \in \Z\,. \ee
Additionally, the theory is unitary for 
$\sigma:=\frac12(t-\tilde t)\in i\BR$, and less obviously so for 
$\sigma\in \BR$. We will not require unitarity in the following, however.

The classical solutions of Chern-Simons theory are flat $G_\C$ connections. 
Indeed, in the limit $t,\tilde t\to \infty$, which corresponds to infinitely 
weak coupling, the partition function 
$\CZ(t,\tilde t) = \int D\calA\, D\ol{\calA}\,e^{-iS}$ is dominated by flat 
connections 
\begin{align} 
\mathcal Z(t,\tilde t) &\sim \sum_{\text{flat $\calA^*$}} \tau(\calA^*)
\exp\bigg[\frac{t}{8\pi i}I_{CS}(\calA^*) + \d(\calA^*)\log t 
+ \sum_{n=2}^\infty \bigg(\frac{8\pi i}{t}\bigg)^{n-1} S_n(\calA^*)\bigg] \\
 &\hspace{.8in} \times 
 \tau(\ol{\calA}^*)
\exp\bigg[\frac{\tilde t}{8\pi i}I_{CS}(\ol{\calA}^*) + \d(\ol{\calA^*})\log t 
+ \sum_{n=2}^\infty \bigg(\frac{8\pi i}{\tilde t}\bigg)^{n-1} 
S_n(\ol{\calA}^*)\bigg]
\notag\end{align}
where $\tau(\calA)^{-2}$ is a Ray-Singer torsion twisted by the flat 
connection $\calA$ and the $S_n$ are ``higher-loop'' topological invariants.
For $G_\C=SL(2,\C)$ and $\calA^*$ the hyperbolic flat connection,
such an asymptotic expansion at weak coupling played a central role in the 
generalized Volume Conjecture \cite{Gu}.

\subsection{A singular limit}
\label{sub.singular.limit}

At present we are interested in a very different limit in complex 
Chern-Simons theory, namely $(t,\tilde t) \to (2k,0)$, or equivalently 
$\sigma \to k$ with $k\in \Z$ held fixed. This is a singular limit rather 
than a weak coupling limit. We propose
\begin{conjecture}
\label{conj.CS}
In the limit $(t,\tilde t) \to (2k,0)$, the partition function of complex 
Chern-Simons theory has an asymptotic expansion
\be 
\label{CSexp} 
\mathcal Z \sim \sum_{\text{flat $\calA^*$}} 
\tau_k(\calA^*) \exp\bigg[\frac1{k\hbar}I_{CS}(\calA^*) 
+ \d(\calA^*)\log(k\hbar) + \sum_{n=2}^\infty \hbar^{n-1} 
S_{n,k}(\calA^*)\bigg]\,, 
\ee
where $\hbar = 2\pi i\, \tilde t/t= 2\pi i(\frac{k-\sigma}{k+\sigma})$ and
 $I_{CS}$ is the holomorphic classical Chern-Simons action. Moreover, if 
$M=S^3\backslash K$ is a hyperbolic knot complement, $G_\C=SL(2,\C)$, and 
$\calA^*$ is the hyperbolic flat connection on $M$, then 
$\d(\calA^*)=-\frac32$ and
the series in the Quantum Modularity Conjecture \eqref{eq.QMF} at 
$\zeta = e^{2\pi i\alpha} = e^{2\pi i/k}$ is
\be 
\label{CSmod} 
\phi_{K,\zeta}(\hbar) = \tau_k (\calA^*) 
\exp\bigg[\sum_{n=2}^\infty \hbar^{n-1} S_{n,k}(\calA^*)\bigg]\,.
\ee
\end{conjecture}
Note that, by definition, $I_{CS}(\calA^*)$ already equals the complex 
hyperbolic volume $V$, so the exponential term $\exp(\frac{1}{k\hbar}V)$ 
already matches on the right hand side of Equations~\eqref{CSexp} and 
\eqref{eq.QMF}.

The existence of the expansion \eqref{CSexp} is physically far from obvious. 
One explanation for \eqref{CSexp} comes from
the so-called 3d-3d correspondence \cite{Yamazaki-3d, DGG}. 
An extension of the original correspondence  relates 
Chern-Simons theory at level $k$ on $M$ to the supersymmetric partition 
function of an associated 3d $\mathcal N=2$ theory $T[M]$ on a lens space  
$L(k,1)\simeq S^3/\Z_k$ \cite{CJ, D-levelk}. The lens space is a $\Z_k$ 
orbifold of a 
sphere, whose geometry has been ellipsoidally deformed such that the ratio 
of minimum to maximum radii is $b=\sqrt{\tilde t/t} = 
\sqrt{\frac{k-\sigma}{k+\sigma}}$. It is well known that as $b\to 0$ the 
partition function of $T[M]$ on a sphere $L(1,1)\simeq S^3$ has an expansion 
of the form \eqref{CSexp}, whose leading exponential term is 
$\frac1\hbar I_{CS}(\calA^*)$, see~\cite{DG-Sdual, Yamazaki-3d, DGG}.
One then expects the $L(k,1)$ partition function to have a similar expansion 
as $b\to 0$, with leading term $\frac1{k\hbar} I_{CS}(\calA^*)$, just as in  
\eqref{CSexp}.

There are also some preliminary hints that the 
existence and structure of \eqref{CSexp} may be explained using 
electric-magnetic duality in four-dimensional Yang-Mills theory, with 
Chern-Simons theory on its boundary along the lines of \cite{Witten-q, Witten-5}. 
Indeed, the electric-magnetic duality group $SL(2,\mathbb Z)$ can relate a 
singular limit such as $(t,\tilde t)\to(2k,0)$ to a more standard 
weak-coupling limit. Electric-magnetic duality has been linked to modular 
phenomena in the past \cite{VafaWitten}, and it is tempting to believe that 
it could provide a physical basis for Quantum Modularity as well.
We aim to explore this further in the future.

\subsection{State integrals}

Complex Chern-Simons theory has not yet been made mathematically rigorous 
as a full TQFT, in contrast to Chern-Simons theory with a compact gauge
group~\cite{Witten-jones} and the Reshetikhin-Turaev construction~\cite{RT}. 
Nevertheless, there exist state-integral 
models, based on ideal triangulations, that provide a definition of 
complex Chern-Simons partition functions for a certain class of 3-manifolds 
\cite{D-levelk, AK}. These state-integral models generalize earlier work 
\cite{Hikami, DGLZ, Dimofte-QRS, AK-TQFT, AK-new} that computed Chern-Simons 
partition functions at level $k=1$.

In the present paper, we use the asymptotic expansion of these state 
integrals in the limit $(t,\tilde t)\to(2k,0)$ to motivate the
definition of the power series $\phi_{\gamma,\zeta}(\hbar)$ given in 
Section~\ref{sec.def}.

Before we discuss the state-integral $\CZ_\gamma$ associated to a 
Neumann-Zagier datum $\ga$ of an ideal triangulation, it is worth mentioning
that convergence of the state-integral  requires certain positivity 
assumptions, which are satisfied when the ideal triangulation supports a 
strict angle structure. This is discussed at length in \cite{AK, D-levelk}. 
In the rest of this section, we will assume that the background ideal 
triangulation admits such a structure. Although positivity is required for 
the convergence of the state-integral, the formula that we will obtain for 
its asymptotic expansion makes sense without any positivity assumptions.

It was shown in \cite{Dimofte-QRS, DG} that a Neumann-Zagier datum
$\gamma = (\mb A,\mb B,\nu,z,f,f'')$ 
with non-degenerate matrix $\mb B$ leads to a state-integral partition 
function for $\SL(2,\BC)$ Chern-Simons at level $k=1$, given by
\begin{align} 
\CZ_\gamma&= \frac{1}{\sqrt{\det\mb B}}(i)^{f\mb B^{-1}\nu}
  e^{\frac1k(\frac\hbar8-\frac{\pi^2}{8\hbar})f\mb B^{-1}\nu} \int \frac{d^NZ}{
(-2\pi i \hbar)^{\frac N2}}e^{-\frac1{2k}(1+\frac{2\pi i}{\hbar})Z\mb B^{-1}\nu
+\frac1{2k\hbar}Z\mb B^{-1}\mb AZ} \prod_{i=1}^N\CZ_\hbar[\Delta](Z_i) \,,
\label{SS1}
\end{align}
where the integral runs over some mid-dimensional contour in the space 
$\BC^N$ parametrized by $Z=(Z_1,...,Z_N)$, and $\CZ_\hbar[\Delta](Z_i)$ is a 
quantum dilogarithm function associated to every tetrahedron, given for 
$\Re(\hbar)<0$ by
\be 
\CZ_\hbar[\Delta](Z_i) = \prod_{r=0}^\infty 
\frac{1-e^{(r+1)\hbar-Z_i}}{1-e^{r \frac{4\pi^2}{\hbar})
-\frac{2\pi i}{\hbar}Z_i}}  = \frac{(e^{Z_i+\hbar},
e^\hbar)_\infty}{(e^{\frac{2\pi i}{\hbar}Z_i},
e^{\frac{4\pi^2}{\hbar}})_\infty}\,.
\ee
This function has an asymptotic expansion as $\hbar\to 0$,
\begin{align}  
\CZ_\hbar[\Delta](Z_i) &\sim (e^{Z+\hbar},e^\hbar)_\infty \sim \exp\sum_{n=0}^\infty 
\frac{\hbar^{n-1}}{n!} B_n(1)\Li_{2-n}(e^{-Z}) \\
& \hspace{.3in}  = \frac1\hbar \Li_2(e^{-Z}) + \frac12\Li_1(e^{-Z}) 
+ \frac\hbar{12}\Li_0(e^{-Z}) + \ldots\,.
\notag
\end{align}
Using $\Li_2(e^{-(Z^*+\delta Z)}) = \Li_2(e^{-Z^*}) 
+ \log(1-e^{-Z^*})\delta Z + ...$, it
follows that at leading order in $\hbar$, the integrand of \eqref{SS1} has 
critical points at
\be 
\label{CScrit} 
\mb A Z^* +\mb B\log(1-e^{-Z^*}) = i\pi \nu\,, 
\ee
which are a logarithmic version of the gluing equations. In particular, 
if $\gamma$ is a positive Neumann-Zagier datum, then the equations are 
satisfied by $Z^*_i=\log (z_i)$ for all $i=1,\dots,N$. By performing formal 
Gaussian integration around this ``geometric'' critical point, order by 
order in the formal parameter $\hbar$, we obtained in \cite{DG} a 
diagrammatic formula for the series $\phi_{K,1}(\hbar)$.

The actual contour of integration appropriate for \eqref{SS1} and its 
level-$k$ generalization has been carefully described in 
\cite{AK-TQFT, D-levelk, AK}. We emphasize, however, that in order to 
perform a formal perturbative expansion around a given critical point, 
a choice of contour is irrelevant.

The level-$k$ generalization of the state integral, as developed in 
\cite{D-levelk}, reads
\begin{align} 
\CZ_\gamma^{(k)}&= \frac{1}{k^N\sqrt{\det\mb B}}\zeta^{\frac14f\mb B^{-1}\nu}
  e^{\frac1k(\frac\hbar8-\frac{\pi^2}{8\hbar})f\mb B^{-1}\nu}  \notag\\
  &\hspace{-.3cm}\times \hspace{-.2cm} \sum_{m\in(\BZ/k\BZ)^N} 
\int \frac{d^NZ}{(-2\pi i \hbar)^{\frac N2}}
(-\zeta^{\frac12})^{m\mb B^{-1}\mb A m} 
e^{-\frac1{2k}(1+\frac{2\pi i}{\hbar})Z\mb B^{-1}\nu+\frac1{2k\hbar}Z\mb B^{-1}\mb AZ} 
\prod_{i=1}^N \CZ_\hbar^{(k)}[\Delta](Z_i,m_i)
\,,\hspace{-.5cm}\label{SSk}
\end{align}
where $\zeta=e^{\frac{2\pi i}{k}}$ as usual, and $(-\zeta^{\frac12})^C$ is 
understood as $e^{(\frac{i\pi}{k}-i\pi)C}$ for any $C$; and
\be 
\CZ_\hbar^{(k)}[\Delta](Z_i,m_i)  = 
\frac{\big(\zeta^{m_i+1}e^{\frac\hbar k-\frac {Z_i}k};\zeta 
e^{\frac\hbar k}\big)_\infty}
{\big(\zeta^{-m_i}e^{-\frac{2\pi i}{k\hbar}{Z_i}};\zeta^{-1}
e^{\frac{4\pi^2}{k\hbar}}\big)_\infty}
\; =\hspace{-.2cm}
 \prod_{\mbox{$0\leq s,t< k \atop  s-t\,\equiv\, m_i\,(\text{mod}\,k)$}} 
\hspace{-.5cm}\CZ_\hbar[\Delta]\Big(\frac {Z_i}k+ \hbar\frac s k 
+2\pi i \frac tk\Big)\,.   
\ee
Recall some well-known facts about the asymptotic expansion of the
quantum dilogarithm that can be found, for instance, 
in~\cite[Sec. II.D]{Za:dilogarithm}.

\begin{lemma}
\label{lem.AQDL}
We have:
$$
\log \, (q x;q)_\infty = -\sum_{n=1}^\infty \frac{q^n x^n}{n(1-q^n)} \sim
\sum_{n=0}^\infty \hbar^{n-1} \frac{(-1)^n B_n}{n!} {\rm Li}_{2-n}(x)
$$
when $q=e^\hbar$ and $\hbar \rightarrow 0$,
where $B_n$ is the $n$-th Bernoulli number and ${\rm Li}_n(x)=\sum_{k=1}^\infty
x^k/n^k$ is the $n$-th polylogarithm. Since 
$x \partial_x {\rm Li}_{2-n}(x)={\rm Li}_{2-n-1}(x)$, it follows that 
$$
\log \, (e^\hbar e^{-(u + w)}; e^\hbar)_\infty \sim 
\exp\left( \sum_{n,k=0}^\infty \hbar^{n-1} 
\frac{(-1)^{n+k} B_n}{n!k!} {\rm Li}_{2-n-k}(e^{-u})w^k
\right)
$$
\end{lemma}

The asymptotic expansion of $\CZ_\hbar^{(k)}$ follows from its product 
representation,
\begin{align} 
\label{psiexp}
\CZ_\hbar^{(k)}[\Delta](Z_i,m_i)  &\sim \prod_{s=0}^{k-1} 
(\zeta^{m_i-s}e^{(1-\frac sk)\hbar-\frac1kZ_i};e^\hbar)_\infty  \notag \\ &
\sim \exp\sum_{n=0}^\infty\sum_{s=1}^{k} \frac{\hbar^{n-1}}{n!}B_n
\big(\tfrac sk\big) \Li_{2-n}(\zeta^{m_i+s}e^{-Z_i/k}) \\ 
& =  \exp\bigg[\frac1{k\,\hbar}\Li_2(e^{-Z_i}) - \sum_{s=1}^{k}
\Big(\frac12-\frac sk\Big)\Li_1(\zeta^{m_i+s}e^{-Z_i/k}) + \ldots\bigg]\, \notag
\end{align}
(where we have substituted $s$ by $k-s$ in the second sum).
Notably, the leading asymptotic $\frac1k\Li_2(e^{-Z_i})$ is independent of 
$m_i$. Indeed, this remains true for the entire integrand in \eqref{SSk}.

The critical points $Z^*$ of the integrand at order $\hbar^{-1}$ simply 
satisfy the standard gluing equation \eqref{CScrit}. Let us assume that the 
Neumann-Zagier has all $z_i$ strictly in the upper half-plane and focus on 
the geometric critical point $Z^* = \log(z)$. The value of the integrand at 
the critical point, at order $\hbar^{-1}$, then becomes (after some 
manipulation)
\be 
\exp\frac1 {k\,\hbar}\bigg[-\frac12(Z^*-i\pi f)\cdot((Z^*)''+i\pi f'') 
+ \sum_{i=1}^N \Li_2(e^{-Z_i^*})\bigg]\,, \label{CSvol} 
\ee
where $(Z^*)'' = \log(1-e^{-Z^*}) = \log z''$.
The quantity \eqref{CSvol} appears to agree with the complex hyperbolic 
volume of a manifold $M$ with Neumann-Zagier datum $\gamma$, modulo 
$\pi^2/6$ \cite{DG}, though knowing this is unnecessary for obtaining the 
series $\phi_{\gamma,\zeta}$. On the other hand, it is crucial for 
our computation that the value at the leading-order saddle point is 
independent of $m$ --- so all terms in the sum over $m$ contribute equally 
to the higher-order asymptotics.

By using \eqref{psiexp}, or (better) the double series expansion around the 
critical point $Z^*$,
\be 
\CZ_\hbar^{(k)}[\Delta](Z_i^*+\delta Z_i,m_i) \sim \exp \sum_{n=0}^\infty
\sum_{j=0}^\infty \frac{\hbar^{n-1}(-\delta Z_i)^j}{k^j\,n! j!} 
\sum_{s=1}^k B_n\big(\tfrac sk\big) \Li_{2-n-j}(\zeta^{m_i+s}e^{-Z_i^*/k})\, 
\ee
a saddle-point approximation or formal Gaussian integration of \eqref{SSk} 
leads immediately to the definition of $\phi_{\gamma,\zeta}(\hbar)$ in Section 
\ref{sec.def}. Indeed, in the finite-dimensional Feynman calculus, the 
propagator $\Pi$ is the inverse of the Hessian matrix, appearing at order 
$\hbar^{-1}$ in the exponent of \eqref{SSk} as 
$\frac1k \delta Z\cdot(-\mb B^{-1}\mb A+\Delta_{(z^*)'})\cdot \delta Z$; 
while each vertex factor $\Gamma_i^{(j)}$ is the coefficient of 
$(\delta Z_i)^j$ in the exponent of \eqref{SSk}.

\subsection{Derivation of the torsion}
\label{sub.derivation.torsion}

To illustrate how the formal Gaussian integration works, let us derive the 
$k$-twisted torsion or ``1-loop invariant'' $\tau_{\gamma,k}$ of 
\eqref{eq.tau}, starting from \eqref{SSk}. Let us set 
$\zeta=e^{\frac{2\pi i}{k}}$ and $\theta_i = z_i^{1/k} = e^{Z_i^*/k}$ as usual, 
and work at fixed $m\in (\BZ/k \BZ)^N$ to start.

There are two contributions to the torsion. First there is the integrand 
itself, evaluated at $Z=Z^*$, keeping only terms of order $\hbar^0$ in the 
exponent:
\be  
\frac{\zeta^{\frac14 f\mb B^{-1}\nu}}{k^N (-2\pi i \hbar)^{\frac N2}
\sqrt{\det \mb B}} \times (-\zeta^{\frac12})^{m\mb B^{-1}\mb Am}
e^{-\frac1{2k}Z^*\mb B^{-1}\nu} \prod_{i=1}^N \prod_{s=1}^{k} 
(1-\zeta^{m_i+s}e^{-Z_i^*/k})^{\frac12-\frac sk} \,, 
\ee
where we have used that $\Li_1(x) = -\log(1-x)$. Second, there is the 
determinant of the Hessian, coming from the leading-order Gaussian integration,
\be \big[\det  \tfrac{1}{2\pi k\hbar}
(-\mb B^{-1}\mb A + \Delta_z)\big]^{-\frac12}\,. 
\ee
Combining these terms, using $\mb Af+\mb Bf''=\nu$ and 
$\mb AZ^*+\mb B(Z^*)''=i\pi \nu$ to rewrite $Z^*\mb B^{-1}\nu$ as 
$Z^*\cdot f'' - (Z^*)''\cdot f+i\pi f\mb B^{-1}\nu$, and observing that 
$\prod_{i=1}^N\prod_{s=1}^k(1-\zeta^{m_i+s}e^{-Z_i^*/k})^{\frac12} 
= \prod_{i=1}^N (1-e^{-Z_i^*})^{\frac12} = (\det \Delta_{z''})^{\frac12}$,
we arrive at
\be 
\tau_{\gamma,k} = \frac{1}{(ik)^{\frac N2}
\sqrt{\det(\mb A\Delta_z''+\mb B\Delta_{z'}^{-1})z^{f''/k}z''{}^{f/k}}} 
(-\zeta^{\frac12})^{m\mb B^{-1}\mb Am}
\prod_{i=1}^N \prod_{s=0}^{k-1} (1-\zeta^{m_i-s}\theta_i^{-1})^{s/k}\,.
\ee
The product may be manipulated further using
\begin{align*} & \prod_{s=0}^{k-1}(1-\zeta^{m_i-s}\theta_i^{-1})^{\frac sk} 
= \prod_{s=-m_i}^{k-1-m_i}(1-\zeta^{-s}\theta_i^{-1})^{\frac {s+m_i}{k}}  \\
&\hspace{.8in}
= z_i''{}^{\frac {m_i}{k}} \prod_{s=k-m_i}^{k-1}
(1-\zeta^{-s}\theta_i^{-1})^{\frac sk-1}\prod_{s=0}^{k+m_i}
(1-\zeta^{-s}\theta_i^{-1})^{\frac sk} = z_i''{}^{\frac {m_i}{k}} 
D_k^*(\theta_i^{-1})^{\frac 1k} (\zeta\theta_i^{-1};\zeta)_{m_i}^{-1}\,.
\end{align*}
Finally, setting 
$\prod_i z_i''{}^{\frac {m_i}k} = \exp[\frac1k (Z^*)''\cdot m] 
= \exp[\frac {i\pi}k m\mb B^{-1}\nu-\frac 1k m\mb B^{-1}\mb AZ]
=\zeta^{\frac12 m\mb B^{-1}\nu}\theta^{-\mb B^{-1}\mb Am}$,
and summing the whole expression over $m\in (\BZ/k\BZ)^N$, we recover 
\eqref{eq.tau}.

\subsection{Ambiguities}
\label{sub.ambiguities}

The state integral \eqref{SSk} has an intrinsic multiplicative ambiguity 
\cite[Eqn. 5.8]{D-levelk}. Namely, it is only defined modulo multiplication 
by factors (at worst) of the form
\be  
\big(e^{\tfrac{\pi^2}{6k\hbar}}\big)^{a_1}  
\big(\zeta^{\frac{1}{24}}\big)^{a_2} 
\big(e^{\tfrac{\hbar}{24k}}\big)^{a_3}\,,\qquad a_1,a_2,a_3\in \Z\,. 
\ee
The second and third factors affect $\tau_{k}$ and $S_{2,k}$, respectively, 
in the asymptotic expansion. Higher-order terms in the expansion are 
unaffected. These ambiguities in the state integral are consistent with 
those discovered experimentally for $\phi_{\gamma,\zeta}$, as discussed in 
Section \ref{sub.inv}.


\section{Computations}
\label{sec.compute}

\subsection{How the data was computed}
\label{sub.conv}

We use the Rolfsen notation for knots~\cite{Rf}. {\tt SnapPy} computes the
Neumann-Zagier matrices of default ideal triangulations of the knots
below, as well as their exact shapes and trace fields (computed for
instance from the Ptolemy module of {\tt SnapPy})~\cite{ptolemy,snappy}.

Given a Neumann-Zagier datum, the 2 and 3-loop invariants at level $k$
are algebraic numbers, elements of the field $F_{K,k}=F_K(\zeta_k)$, 
where $F_K$ is the trace field of $K$ and $\zeta_k=e^{\frac{2\pi i}{k}}$.
However, these numbers are obtained by sums of algebraic numbers in a 
much larger number field. Moreover, the 1-loop invariant at level $k$
already contains a $k$-th root of elements of $F_{K,k}$. 
This makes exact computations impractical. To produce the
interesting factorization of Equation~\eqref{eq.phi0factor}, and keeping
in mind the ambiguities of Section~\ref{sub.ambiguities} we proceed as
follows. We know that 
$x_{k,\ell}=\tau^k_k/(\tau^k_1 \zeta_{24k}^\ell) \in F_{K,k}$
for some natural number $\ell$. Given this, we compute the numerical value
$x^{\mathrm{num}}_{k,\ell}$ of 
$x_{k,\ell}$ (for several values of $\ell$) and find a value of $\ell$ for which 
there is an element $x^{\mathrm{exact}}_{k,\ell}$ of $F_{K,k}$ which is 
reasonably close to our element. We accomplish this by the LLL
algorithm~\cite{LLL}. The Quantum Modularity Conjecture asserts that the 
exact element of $F_{K,k}$ should to have the form:
\be
\label{eq.xe}
x^{\mathrm{exact}}_{k,\ell} = \varepsilon_{K,k} \beta_{K,k}^k
\ee
for $\varepsilon_{K,k} \in \calO_{F_K(\zeta_k)}^{\times}$ (an algebraic unit)
and $\beta_{K,k} \in F_K(\zeta_k)^\times$. To find $\epsilon_{K,k}$ and 
$\beta_{K,k}$, factor the fractional ideal 
\be
\label{eq.xpi}
x^{\mathrm{exact}}_{k,\ell} \calO_{F_K(\zeta_k)} = 
\mathfrak{p}_1^{e_1} \dots \mathfrak{p}_r^{e_r} 
\ee
into a product of prime ideals $\mathfrak{p}_i$, $i=1,\dots,r$.
If all ramification exponents $e_i$ are divisible by $k$, and if the
prime ideals are principal $\mathfrak{p}_i=(\wp_i)$ for $\wp_i \in
F_{K,k}^\times$ (the latter happens when the ideal class group of 
$F_{K,k}$ is trivial), then we define
$$
\beta_{K,k}=\prod_{i=1}^r \wp_i^{\frac{e_i}{k}} \in F_{K,k}^\times, \qquad
\varepsilon_{K,k}= x^{\mathrm{exact}}_{k,\ell}/\beta_{K,k}^k \,.
$$
It follows that $\varepsilon_{K,k}$ is a unit, and that \eqref{eq.xe}
holds. This gives us strong confidence that $x^{\mathrm{exact}}_{k,\ell}$
is the correct element, and that the computation is correct.

In practice, we have used a {\tt Mathematica} program to compute 
$x^{\mathrm{num}}_{k,\ell}$ and $x^{\mathrm{exact}}_{k,\ell}$, and a {\tt Sage}
program (that uses internally {\tt pari-gp}) to compute the ideal 
factorization \eqref{eq.xpi}.



\subsection{A sample computation}
\label{sub.sample}

Let us illustrate our method of computation in detail with one example, the
$5_2$ knot with $k=7$. $5_2$ is a hyperbolic knot with trace
field $F_{5_2}=\BQ(\a)$ where $\a=0.8774 \dots - 0.7448 \dots i $ is a root of 
$$
x^3 -x^2 + 1 = 0 \,.
$$
$F_{5_2}$ is of type $[1,1]$ with discriminant $-23$. With the notation of
the previous section, and with $\ell=6$, we can numerically compute
(with 500 digits of accuracy) 
$$
x^{\mathrm{num}}_{7,6}=
-235162149.63362564574 \dots - 40898882. 99885002594 \dots i
$$
Fitting with LLL guesses the element of $F_7=F(\zeta_7)$
{\small
\begin{align*}
x^{\mathrm{exact}}_{7,6} &=
 -42626237 - 31168064 \a + 54414583 \a^2 +  
   (3905252 - 48974302 \a + 103510169 \a^2)  \zeta_7 + \\ & 
   (91608760 - 23650188 \a + 97210659 \a^2)  \zeta_7^2 +  
   (158817619 + 22023535 \a + 44886912 \a^2)  \zeta_7^3 - \\ & 
   (-149267670 - 54779388 \a + 17355247 \a^2)  \zeta_7^4 - 
   (-80916790 - 45810663 \a + 37182537 \a^2)  \zeta_7^5
\end{align*}
}
How can we trust this answer? 
We can compute the norm $N(x^{\mathrm{exact}}_{7,6})$ of $x^{\mathrm{exact}}_{7,6}$
(that is, the product of all Galois conjugates) and find out that:
$$
N(x^{\mathrm{exact}}_{7,6}) = 43^{14} \cdot 6007111235971721^7 \,.
$$
It is encouraging that the above norm is the seventh power of an integer. 
But even better is the fact that we can factor the ideal generated by 
the above element as follows:
$$
(x^{\mathrm{exact}}_{7,6})= (\wp_{43})^{14} \cdot (\wp_{6007111235971721})^7
$$
where 
\begin{align*}
\wp_{43} &= (\a - 1) \zeta_7^5 + \a \zeta_7^2 + \a \\
\wp_{6007111235971721} &= 
(4 \a^2 + 6 \a - 7) \zeta_7^5 + 
(5 \a^2 + 4 \a - 3) \zeta_7^4 + 
(8 \a^2 + \a - 8) \zeta_7^3 + \\ & \quad\,\,
(3 \a^2 + 5 \a - 6) \zeta_7^2 + 
(2 \a^2 + \a - 5) \zeta_7 + 6 \a^2 - 2 \a - 2
\end{align*}
are primes of norm $43$ and $6007111235971721$ (a prime number), 
respectively.
If we define $\beta_{7}=\wp_{43}^2 \cdot \wp_{6007111235971721} \in F_7$ and
$\varepsilon_{7}= x^{\mathrm{exact}}_{7,6}/\beta_{7}^7$
it follows that 
$$
x^{\mathrm{exact}}_{7,6} = \varepsilon_{7} \beta_{7}^7
$$
where $\varepsilon_{7} \in F_7^\times$ is a unit, given explicitly by the rather
long expression:
\begin{align*}
\varepsilon_{7} &=
 (318981244103 \a^2 + 40488788528803 \a + 30382313828818) \zeta_7^5 + 
\\ & \quad\,\,
(-52797766935255 \a^2 + 38212176617858 \a + 58931813581928) \zeta_7^4 + 
\\ & \quad\,\,
(-29477571352182 \a^2 - 1263424293533 \a + 15843777055057) \zeta_7^3 + 
\\ & \quad\,\,
(13260713424737 \a^2 + 18581482784028 \a + 6470257562608) \zeta_7^2 + 
\\ & \quad\,\,
(-29079808246903 \a^2 + 49225269181062 \a + 53729902713340) \zeta_7 - 
\\ & \quad\,\,
52974788170701 \a^2 + 15742594165404 \a + 42070901450997  \,.
\end{align*}
This is an answer that we can trust. There is an additional invariance
property of the above unit under the Galois group of $\BQ(\zeta_7)/\BQ$,
discussed in detail in~\cite{calegari}.


\section{Data}
\label{sec.data}

\subsection{The $4_1$ knot}
\label{sub.41}

The $4_1$ knot is the simplest hyperbolic knot with
volume $2.0298\dots$ with $2$ ideal tetrahedra and trace
field $F_{4_1}=\BQ(\a)=\BQ(\sqrt{-3})$ where $\a=e^{2 \pi i/6}$ is a root of 
$$
x^2 - x + 1 = 0\,.
$$
$F_{4_1}$ is of type $[0,1]$ with discriminant $-3$.

The default {\tt SnapPy} triangulation of $4_1$ generates several 
Neumann-Zagier data. Most are $\Z$-nondegenerate; for example 
\be
\ga\,:\quad  
 \mb A = \begin{pmatrix}-2&1\\-1&1\end{pmatrix}\,,\;\;
\mb B= \begin{pmatrix} -1&2\\-1&1 \end{pmatrix}\,,\;\; 
\nu = (0,0)\,,\;\; z = (\a,\a)\,,\;\; f=(0,1)\,,\;\;f''=(1,0)\,
\ee
is $\BZ$-nondegenerate. The 1-loop invariant at $k=1$ and its
norm is given by
\be
\label{eq.1loop41}
\begin{array}{|c|l|l|} \hline
\text{knot} & \tau_1^{-2} & N(\tau_1^{-2}) \\ \hline
4_1 & 2 \a - 1 & 3 \\ \hline
\end{array} 
\ee
The norm of the 1-loop of $4_1$ at level $k$ is given in~\eqref{t.41norm}.
{\small
\be
\label{t.41norm}
\begin{array}{|c|l|} \hline
k & N(\tau_k/\tau_1) \qquad \text{for $4_1$} \\ \hline
1 & 1  \\ \hline 
2 & 3  \\ \hline
4 & 11^2  \\ \hline
5 & 3^4 \cdot 29^2  \\ \hline
7 & 39733^2  \\ \hline 
8 &  3^4 \cdot 383^2 \\ \hline
10 & 19289^2 \\ \hline 
11 & 3^{10} \cdot 463^2 \cdot 128237^2 \\ \hline 
13 & 13339^2 \cdot 13963^2 \cdot 130027^2 \\ \hline 
14 & 3^6 \cdot 419^2 \cdot 4451^2 \\ \hline 
16 & 97^2 \cdot 418140719^2 \\ \hline 
17 & 3^{16} \cdot 170239^2 \cdot 377615549357^2 \\ \hline 
19 & 571^2 \cdot 2851^2 \cdot 27513329^2 \cdot 83702994059^2 \\ \hline 
20 & 3^8 \cdot 59^2 \cdot 975911939^2 \\ \hline 
22 & 131^2 \cdot 14783^2 \cdot 39667^2 \cdot 92927^2 \\ \hline 
\end{array}
\ee
}
In the above table, we avoided the (degenerate) case when $k$ is divisible 
by $3$, since in those cases the trace field contains the third roots of unity.
Notice that the above norms are squares of integers. This exceptional
integrality may be a consequence of the fact that $4_1$ is amphicheiral.

Next, we give some sample computations of the 
factorization~\eqref{eq.phi0factor}. In this and the next sections,
more data has been computed (even for
non-prime levels $k$), but only a sample will be presented here. 
Throughout this section, $\wp_{n}$ will denote a prime
in $\calO_{F_k}$ of norm $n$, a prime power. 

\noindent 
For $k=2$ 
we have
{\small
\begin{align*}
\varepsilon_2 &=\a \\
\beta_2 &= \wp_{3}  \\
\wp_{3} &= 2 \a - 1
\end{align*}
}

\noindent 
For $k=4$ and $\zeta=\zeta_4$ 
we have
{\small
\begin{align*}
\varepsilon_4 &= (-2 \a + 1) \z - 2 \\
\beta_4 &= \wp_{11^2}  \\
\wp_{11^2} &= (-4 \a + 2) \z + 1
\end{align*}
}

\noindent 
For $k=5$ and $\zeta=\zeta_5$ we have
{\small
\begin{align*}
\varepsilon_5 &=
(4 \a - 5) \zeta^3 + (8 \a - 5) \zeta^2 + 7 \a \zeta + 2 \a + 3 \\
\beta_5 &= \wp_{3^3} \cdot \wp_{29^2} \\
\wp_{3^3} &=2 \a - 1 \\
\wp_{29^2} &=-2 \zeta^3 + (-\a - 1) \zeta^2 - \a \zeta + \a - 2
\end{align*}
}

\noindent 
For $k=7$ and $\zeta=\zeta_7$ we have
{\small
\begin{align*}
\varepsilon_7 &=(60 \a + 115) \zeta^5 + (239 \a + 90) \zeta^4 
+ (390 \a - 61) \zeta^3 + (415 \a - 240) \zeta^2 +
\\ & \quad\,\,
 (300 \a - 300) \zeta + 114 \a - 186 \\
\beta_7 &= \wp_{39733,1} \cdot \wp_{39733,2} \\
\wp_{39733,1} &= \a \zeta^5 + (\a - 2) \zeta^4 + (2 \a - 1) \zeta^3  
+ (\a - 2) \zeta^2 + (2 \a - 1) \zeta + 2 \a - 2 \\
\wp_{39733,2} &=\a \zeta^5 - \zeta^4 + \a \zeta^3 - \zeta^2 
+ (2 \a - 1) \zeta + 2 \a - 2
\end{align*}
}

\noindent 
For $k=8$ and $\zeta=\zeta_8$ we have
{\small
\begin{align*}
\varepsilon_8 &=
(-12 \a + 36) \z^3 + (17 \a + 17) \z^2 + (36 \a - 12) \z + 34 \a - 34 \\
\beta_8 &= \wp_{3^2,1} \cdot \wp_{3^2,2} \cdot \wp_{383^2} \\
\wp_{3^2,1} &=\z^3 - \z^2 - \a + 1 \\
\wp_{3^2,2} &= -\z^3 + \a \z + 1 \\
\wp_{383^2} &= (-2 \a + 2) \z^3 + (3 \a - 1) \z^2 + (\a + 1) \z - \a + 3
\end{align*}
}

\noindent 
For $k=10$ and $\zeta=\zeta_{10}$ we have
{\small
\begin{align*}
\varepsilon_{10} &=
(-3 \a + 7) \z^3 + (6 \a - 4) \z^2 + (\a + 1) \z + 9 \a - 10 \\
\beta_{10} &= \wp_{19289^2}  \\
\wp_{19289^2} &= (2 \a - 5) \z^3 + (3 \a + 3) \z^2 + (-\a + 6) \z + 9 \a + 2
\end{align*}
}

\noindent
For $k=11$ and $\zeta=\zeta_{11}$ we have
{\small
\begin{align*}
\varepsilon_{11} &= (-353875255116707 \a + 117872583555117) \zeta^9 +
\\ & \quad\,\,
(-688446384174845 \a + 401155759804328) \zeta^8 +
\\ & \quad\,\,
(-897492189704312 \a + 759910737126284) \zeta^7 +
\\ & \quad\,\,
(-914641448990872 \a + 1080237360407592) \zeta^6 +
\\ & \quad\,\,
(-734446808203694 \a + 1260432001194770) \zeta^5 +
\\ & \quad\,\,
(-414120184922386 \a + 1243282741908210) \zeta^4 +
\\ & \quad\,\,
(-55365207600430 \a + 1034236936378743) \zeta^3 + 
\\ & \quad\,\,
(227917968648781 \a + 699665807320605) \zeta^2 + 
\\ & \quad\,\,
(345790552203898 \a + 345790552203898) \zeta + 
\\ & \quad\,\,
260826355539896 \a + 84964196664002
\\
\beta_{11} &= \wp_{3^5,1} \cdot \wp_{3^5,2} \cdot \wp_{463,1} \cdot \wp_{463,2} 
\cdot \wp_{128237^2}\\
\wp_{3^5,1} &= (\a - 1) \zeta^8 + (\a - 1) \zeta^7 + (\a - 1) \zeta^6 
+ (\a - 1) \zeta^5 - \zeta^4 + (\a - 1) \zeta^2 + (\a - 1) \zeta - 1 
\\
\wp_{3^5,2} &=-\zeta^8 - \a \zeta^5 - \zeta^4 + (\a - 1) \zeta^3 - \zeta + \a - 1
\\
\wp_{463,1} &=(\a - 1) \zeta^9 + \a \zeta^8 + \zeta^7 - \zeta^6 
+ (\a - 1) \zeta^5 + \a \zeta^4 + (\a - 1) \zeta^2 + \a \zeta
\\
\wp_{463,2} &=\a \zeta^8 + \zeta^7 + \a \zeta^5 + \zeta^4 
+ (-\a + 1) \zeta^3 + (\a - 1) \zeta^2 + \a \zeta + 1
\\
\wp_{128237^2} &=2 \a \zeta^9 + 2 \a \zeta^8 + \zeta^6 
+ (-2 \a + 1) \zeta^5 + (-\a - 1) \zeta + \a - 1
\end{align*}
}
Some 2 and 3-loop invariants are shown next. 
{\small
$$
\begin{array}{|c|l|} \hline
k & S_{2,k} \qquad \text{for $4_1$} \\ \hline
1 & (-10 + 11 \a)/108 \\ \hline
2 & (-25 + 41 \a)/216 \\ \hline
3 & (-20 + 37 \a)/108 \\ \hline
4 & (-977 + 1855 \a)/4752 + (5 \zeta)/44 \\ \hline
5 & (-14482 + 37559 \a)/78300 + (11 \zeta)/87 - (
 2 (-133 + 11 \a) \zeta^2)/2175 + ((31 - 22 \a) \zeta^3)/2175 \\ \hline
\end{array}
$$
$$
\begin{array}{|c|l|} \hline
k & S_{3,k} \qquad \text{for $4_1$} \\ \hline
1 & -1/54 \\ \hline
2 & -19/216 \\ \hline
3 & -401/1944 \\ \hline
4 & -17783/52272 + 347 (-1 + 2 \a) \zeta/23232 \\ \hline
5 & (-1569081 + 48037 \a)/2838375 + 
48037 (-1 + 2 \a) \zeta/2838375 + 
\\ & 
(-64041 + 64472 \a) \zeta^2/1892250 + 
(-94781 - 1268 \a) \zeta^3/5676750
\\ \hline
\end{array}
$$
}

\subsection{The $5_2$ knot and its partner, the $(-2,3,7)$ pretzel knot}
\label{sub.52}

The $5_2$ knot is a hyperbolic knot with
volume $2.8281\dots$ with $3$ ideal tetrahedra and trace
field $F_{5_2}=\BQ(\a)$ where $\a=0.8774 \dots - 0.7448 \dots i $ is a root of 
$$
x^3 -x^2 + 1 = 0
$$
$F_{5_2}$ is of type $[1,1]$ with discriminant $-23$.

The (mirror image of) the $(-2,3,7)$ pretzel knot is a hyperbolic knot same
volume and trace field as the $5_2$ knot. In fact, the complements of the
two knots can be obtained from the same triple of ideal tetrahedra with
two different face pairing rules. So, we will use $\a$ and $F$ as in 
Section~\ref{sub.52}. The 1-loop invariant at $k=1$ and its norm is given by

\be
\label{eq.1loop52}
\begin{array}{|c|l|l|} \hline
\text{knot} & \tau_1^{-2} & N(\tau_1^{-2}) \\ \hline
5_2 & 3 \a - 2 & - 23 \\ \hline
(-2,3,7) & -6 \a^2 + 10 \a - 4 & - 2^3 \cdot 23 \\ \hline
\end{array} 
\ee
The norm of the 1-loop of the $5_2$ and $(-2,3,7)$ pretzel knots
at level $k$ is given in~\eqref{t.52norm} and~\eqref{t.237norm}
respectively.
{\tiny
\be
\label{t.52norm}
\begin{array}{|c|l|} \hline
k & N(\tau_k/\tau_1) \qquad \text{for $5_2$} \\ \hline
1 & 1 \\ \hline
2 & 11 \\ \hline
3 & 7^2 \cdot 43 \\ \hline
4 & 21377 \\ \hline
5 & 9491 \cdot 1227271 \\ \hline
6 & 709 \cdot 2689 \\ \hline
7 & 43^2 \cdot 6007111235971721 \\ \hline
8 & 17 \cdot 113 \cdot 7537 \cdot 30133993 \\ \hline
9 & 2083098097 \cdot 85444190599483 \\ \hline
10 & 1811 \cdot 4391 \cdot 68626575961 \\ \hline
11 & 363424007 \cdot 793250477933 \cdot 3103695493140688356241 \\ \hline
12 & 420361 \cdot 5976193 \cdot 119577001 \\ \hline
13 & 3^3 \cdot 35023 \cdot 17090197144904885763873428788615162336954707155761 
\\ \hline
14 & 397951 \cdot 4686537997 \cdot 36383829926671291 \\ \hline
15 & 61 \cdot 271 \cdot 5728621 \cdot 5533526674625134504126790661929671 
\\ \hline
16 & 7^2 \cdot 17 \cdot 6709259559659307953 \cdot 10201336785134943810833 
\\ \hline
17 & 37251471121483 \cdot 478394043550588093915627 \cdot 
\\  &
     680530260143787862026942663543619843463116564673 \\ \hline
18 & 4519 \cdot 8815472623 \cdot 178770985453 \cdot 2913137889913 \\ \hline
19 & 12772437704449 \cdot 3508919046521483041 \cdot 
498973019420515924143242422019287 \cdot 
\\ &
     10714512841561797401872096224519192623 \\ \hline 
20 & 281 \cdot 3821 \cdot 3989122481 \cdot 
3748225906180094225903982496437822401 \\ \hline
21 & 2^6 \cdot 7^2 \cdot 211 \cdot 337 \cdot 913753 \cdot 
2082346663352803 \cdot 
\\ &
     17854362817614367334282334028504194189424575574129 \\ \hline
22 & 11 \cdot 6029 \cdot 54583 \cdot 7275369969656838010303 \cdot 
     8746524744220626866965904589334067 \\ \hline
\end{array}
\ee
\be
\label{t.237norm}
\begin{array}{|c|l|} \hline
k & N(\tau_k/\tau_1) \qquad \text{for $(-2,3,7)$ pretzel} \\ \hline
1 & 1 \\ \hline
2 & 2 \cdot \sqrt{2} \\ \hline
3 & 373 \\ \hline
4 & 2^3 \cdot 373 \\ \hline
5 & 7121 \cdot 7951 \\ \hline
6 & 2^3 \cdot 7 \cdot 7537 \\ \hline
7 & 38543 \cdot 215990584223 \\ \hline
8 & 2^6 \cdot 47389590553 \\ \hline
9 & 19^2 \cdot 109 \cdot 357859 \cdot 3981077803 \\ \hline
10 & 2^6 \cdot 11^2 \cdot 971^2 \cdot 1091 \cdot 1151 \\ \hline
11 & 727 \cdot 2272057394576817291015189643460557 \\ \hline
12 & 2^6 \cdot 20467677759464113 \\ \hline
13 & 937 \cdot 6761 \cdot 160967 \cdot 23955361 \cdot 635301473 \cdot 
57335784304171782943 \\ \hline
14 & 2^9 \cdot 1163 \cdot 89392529932786422898277 \\ \hline
15 & 33137687439067819192706277002439756570331 \\ \hline
16 & 2^{12} \cdot 7^4 \cdot 17 \cdot 5963163273069615265031366100433 \\ \hline
17 & 137 \cdot 82399986307 \cdot 3263165781611 \cdot 
\\ &
     39270783190888798960324268124076297625100114717631 \\ \hline
18 & 2^9 \cdot 19 \cdot 776332747 \cdot 464491149268013810443 \\ \hline
19 & 97553069 \cdot 451234687 \cdot 4511912067991298785435699217959 \cdot 
\\ &
     10780714359892164395007021907819650272965937 \\ \hline
20 & 2^{12} \cdot 101 \cdot 181^2 \cdot 58661 \cdot 1310381 \cdot 
311721147290512745903667881 \\ \hline
21 & 2^6 \cdot 5839 \cdot 295429 \cdot 10289973200263 \cdot 
     168245809559535775760775546501360397248599028829 \\ \hline
22 & 2^{15} \cdot 8532271651199678660022917719747178676450107088703587421
\\ \hline
\end{array}
\ee
}

Next, we give some sample computations of the 
factorization~\eqref{eq.phi0factor}. 

\noindent 
For $k=2$ 
we have for $5_2$
{\small
\begin{align*}
\varepsilon_{2} &= -\a^2 + \a
\\
\beta_{3} &= \wp_{11}  \\
\wp_{11} &= \a^2 + \a - 2
\end{align*}
}
and for $(-2,3,7)$, respectively:
{\small
\begin{align*}
\varepsilon_{2} &=\a + 1 
\\
\beta_{2} &= \wp_{2^3}^{1/2} \\
\wp_{2^3} &=2
\end{align*}
}

\noindent 
For $k=3$ and $\zeta=\zeta_3$ 
we have for $5_2$
{\small
\begin{align*}
\varepsilon_3 &= (-4 \a^2 + 2 \a + 4) \z - 4 \a^2 - \a + 1
\\
\beta_3 &= \wp_{7}^2 \cdot \wp_{43} \\
\wp_{7} &=(-\a^2 + 1) \z - \a^2 + \a
\\
\wp_{43} &=2 \z + \a + 1
\end{align*}
}
and for $(-2,3,7)$, respectively:
{\small
\begin{align*}
\varepsilon_3 &= -\a^2 + 1
\\
\beta_3 &= \wp_{373}  \\
\wp_{373} &= (-2 \a^2 + 2 \a) \z - 2 \a^2 + \a + 2
\end{align*}
}

\noindent 
For $k=4$ and $\zeta=\zeta_4$ we have for $5_2$
{\small
\begin{align*}
\varepsilon_{4} &= -2 \a \z - 2 \a^2 + \a + 1
\\
\beta_{4} &= \wp_{21377}  \\
\wp_{21377} &= (4 \a^2 - 2 \a + 1) \z - 4 \a^2 + 2 \a + 2
\end{align*}
}
and for $(-2,3,7)$, respectively:
{\small
\begin{align*}
\varepsilon_{4} &= (2 \a + 2) \z + 3 \a^2 - 2
\\
\beta_{4} &= \wp_{2^3} \cdot \wp_{373}  \\
\wp_{2^3} &= -\z + 1
\\
\wp_{373} &= (-\a^2 - 2 \a + 1) \z + \a^2 - \a
\end{align*}
}

\noindent 
For $k=5$ and $\zeta=\zeta_5$ we have for $5_2$
{\small
\begin{align*}
\varepsilon_{5} &= 
(-\a^2 + 3 \a) \z^3 + (-2 \a^2 + \a) \z^2 + (2 \a^2 - \a) \z - \a^2 + 2 \a + 1
\\
\beta_{5} &= \wp_{9491} \cdot \wp_{1227271} \\
\wp_{9491} &= (-\a^2 + 2) \z^3 + (-\a^2 + \a + 1) \z^2 + \a
\\
\wp_{1227271} &= (2 \a^2 + 1) \z^3 - \z^2 + (-\a^2 + \a) \z + 1
\end{align*}
}
and for $(-2,3,7)$, respectively:
{\small
\begin{align*}
\varepsilon_{5} &= 
(-5 \a - 4) \z^3 + (10 \a^2 - 10 \a - 3) \z^2 + (20 \a^2 - 10 \a - 2) \z 
+ 12 \a^2 - 5 \a
\\
\beta_{5} &= \wp_{7121} \cdot \wp_{7951} \\
\wp_{7121} &=(-\a^2 + \a - 1) \z^3 + (-\a^2 + \a - 1) \z^2 - \z + 2 \a - 1
\\
\wp_{7951} &=(\a^2 - \a + 1) \z^3 + (\a + 1) \z^2 + (\a^2 - 1) \z + \a^2 + 1
\end{align*}
}

\noindent 
For $k=6$ and $\zeta=\zeta_6$ we have for $5_2$
{\small
\begin{align*}
\varepsilon_{6} &= 
(-24 \a^2 - 12 \a + 4) \z - 6 \a^2 + 24 \a + 21
\\
\beta_{6} &= \wp_{709} \cdot \wp_{2689} \\
\wp_{709} &= (\a + 1) \z - 2 \a^2 - 2
\\
\wp_{2689} &= (\a^2 + 2 \a - 3) \z - 3 \a^2 + \a + 3
\end{align*}
}
and for $(-2,3,7)$, respectively:
{\small
\begin{align*}
\varepsilon_{6} &= 
(\a^2 - 1) \z - \a^2 + 1
\\
\beta_{6} &= \wp_{2^6}^{1/2} \cdot \wp_{7} \cdot \wp_{7357} \\
\wp_{2^6} &= 2
\\
\wp_{7} &= (-\a^2 + \a + 1) \z - 1
\\
\wp_{7357} &= (-2 \a^2 - 3 \a + 3) \z + 2 \a^2 + 1
\end{align*}
}

\noindent 
For $k=7$ and $\zeta=\zeta_7$ we have for $5_2$
{\small
\begin{align*}
\varepsilon_{7} &= 
(318981244103 \a^2 + 40488788528803 \a + 30382313828818) \z^5 +
\\ & \quad\,\,  
(-52797766935255 \a^2 + 38212176617858 \a + 58931813581928) \z^4 + 
\\ & \quad\,\,
(-29477571352182 \a^2 - 1263424293533 \a + 15843777055057) \z^3 + 
\\ & \quad\,\,
(13260713424737 \a^2 + 18581482784028 \a + 6470257562608) \z^2 + 
\\ & \quad\,\,
(-29079808246903 \a^2 + 49225269181062 \a + 53729902713340) \z - 
\\ & \quad\,\,
52974788170701 \a^2 + 15742594165404 \a + 42070901450997
\\
\beta_{7} &= \wp_{43}^2 \cdot \wp_{6007111235971721}  \\
\wp_{43} &=(\a - 1) \z^5 + \a \z^2 + \a
\\
\wp_{6007111235971721} &=
(4 \a^2 + 6 \a - 7) \z^5 + (5 \a^2 + 4 \a - 3) \z^4 + (8 \a^2 + \a - 8) \z^3 +
\\ & \quad\,\, 
(3 \a^2 + 5 \a - 6) \z^2 + (2 \a^2 + \a - 5) \z + 6 \a^2 - 2 \a - 2
\end{align*}
}
and for $(-2,3,7)$, respectively:
{\small
\begin{align*}
\varepsilon_{7} &= 
(349 \a^2 + 119 \a - 176) \z^5 + (439 \a^2 - 196 \a - 450) \z^4 + (60 \a^2 - 189 \a - 143) \z^3 + 
\\ & \quad\,\,
(185 \a^2 + 42 \a + 52) \z^2 + (555 \a^2 - 154 \a - 278) \z + 279 \a^2 - 324 \a - 305
\\
\beta_{7} &= \wp_{38543} \cdot \wp_{215990584223}  \\
\wp_{38543} &=
(-\a^2 - \a + 1) \z^5 + (-\a^2 + 1) \z^4 + (-\a^2 - \a + 1) \z^3 + (-2 \a + 1) \z^2 - \a \z - \a
\\
\wp_{215990584223} &=
(-3 \a^2 - 1) \z^5 + (-5 \a - 3) \z^4 + (2 \a^2 - 3 \a - 3) \z^3 + 
\\ & \quad\,\,
(-\a^2 - 1) \z^2 + (-3 \a^2 - 3 \a + 1) \z - 3 \a - 3
\end{align*}
}

\noindent 
For $k=8$ and $\zeta=\zeta_8$ we have for $5_2$
{\small
\begin{align*}
\varepsilon_{8} &= 
(-41580 \a^2 + 32068 \a + 49052) \z^3 + (-4418 \a^2 + 43620 \a + 32476) \z^2 + 
\\ & \quad\,\,
(35332 \a^2 + 29620 \a - 3124) \z + 54385 \a^2 - 1731 \a - 36894
\\
\beta_{8} &= \wp_{17} \cdot \wp_{113} \cdot \wp_{7537} 
\cdot \wp_{30133993}\\
\wp_{17} &= \z - \a^2
\\
\wp_{113} &=(\a^2 - \a) \z^2 - \z + 1
\\
\wp_{7537} &=(\a - 1) \z^3 + (-\a^2 + \a) \z^2 + \a^2 \z + \a
\\
\wp_{30133993} &=
(\a - 1) \z^3 + (-2 \a^2 + 2 \a - 2) \z^2 + (2 \a^2 + 2 \a - 1) \z
\end{align*}
}
and for $(-2,3,7)$, respectively:
{\small
\begin{align*}
\varepsilon_{8} &= 
(-245132 \a^2 + 447868 \a - 364300) \z^3 + (-888194 \a^2 + 1592676 \a - 1214108) \z^2 + 
\\ & \quad\,\,
(-1010964 \a^2 + 1804516 \a - 1352708) \z - 541525 \a^2 + 959295 \a - 698910
\\
\beta_{8} &= \wp_{2^3}^2 \cdot \wp_{47389590553} \\
\wp_{2^3} &=\z^3 + 1
\\
\wp_{47389590553} &=
(-5 \a^2 + 5 \a) \z^3 + (4 \a^2 - \a) \z^2 + (4 \a^2 + 2 \a - 3) \z + \a^2
\end{align*}
}

\noindent 
For $k=9$ and $\zeta=\zeta_9$ we have for $5_2$
{\small
\begin{align*}
\varepsilon_{9} &= 
(-4941 \a^2 - 5373 \a) \z^5 + (-12105 \a^2 + 5373 \a) \z^4 + (-13605 \a^2 + 13605 \a) \z^3 + 
\\ & \quad\,\,
(-13680 \a^2 + 10098 \a) \z^2 + (-11889 \a^2 + 15471 \a) \z - 4535 \a^2 + 13605 \a
\\
\beta_{9} &= \wp_{2083098097} \cdot \wp_{85444190599483} \\
\wp_{2083098097} &=
2 \a^2 \z^5 + (\a^2 + 2) \z^4 + \a \z^3 + (\a^2 - 1) \z^2 + \a^2 + 1
\\
\wp_{85444190599483} &=
(3 \a^2 - 5 \a + 1) \z^5 - \a \z^4 + (-2 \a^2 - 3 \a + 1) \z^3 + (2 \a^2 - 4 \a + 1) \z^2 + 2 \a^2 \z - 4 \a - 1
\end{align*}
}
and for $(-2,3,7)$, respectively:
{\small
\begin{align*}
\varepsilon_{9} &= 
(331893396 \a^2 + 165165777 \a - 64446273) \z^5 + 
\\ & \quad\,\,
(76625316 \a^2 + 307221984 \a + 188250822) \z^4 + 
\\ & \quad\,\,
(-214496600 \a^2 + 305525612 \a + 352863266) \z^3 + 
\\ & \quad\,\,
(-73359774 \a^2 + 326036187 \a + 287920791) \z^2 + 
\\ & \quad\,\,
(-329761962 \a^2 + 248164137 \a + 375245217) \z - 
\\ & \quad\,\,
431864865 \a^2 + 54173331 \a + 286988238
\\
\beta_{9} &= \wp_{19^2} \cdot \wp_{109} \cdot \wp_{357859}
\cdot \wp_{3981077803} \\
\wp_{19^2} &=
(-3 \a^2 + 2) \z^5 + (-2 \a - 2) \z^4 + (3 \a^2 - \a - 3) \z^3 + 
\\ & \quad\,\,
(-2 \a^2 + 2 \a + 3) \z^2 + (-3 \a^2 + 1) \z + \a^2 - 3 \a - 2
\\
\wp_{109} &= -\a^2 \z^4 - \a^2 \z - \a^2 + \a
\\
\wp_{357859} &=
(-\a^2 + 2 \a) \z^5 + (\a^2 - \a - 1) \z^4 + (-\a^2 + \a + 1) \z^3 + 
\\ & \quad\,\,
(-\a^2 + \a) \z^2 + (\a^2 - \a - 1) \z - \a^2 + 2
\\
\wp_{3981077803} &=
\z^5 + (-2 \a + 2) \z^4 + (\a^2 - \a + 2) \z^3 + \z^2 + (-\a + 2) \z - \a - 1
\end{align*}
}

\noindent 
For $k=10$ and $\zeta=\zeta_{10}$ we have for $5_2$
{\small
\begin{align*}
\varepsilon_{10} &= 
(-3824672997 \a^2 - 3325045215 \a - 330502768) \z^3 + 
\\ & \quad\,\,
(-2263297486 \a^2 + 2676462965 \a + 3310113266) \z^2 + 
\\ & \quad\,\,
(-3762572681 \a^2 - 400845875 \a + 1841500561) \z + 
\\ & \quad\,\,
100480422 \a^2 + 4731453922 \a + 3514375210
\\
\beta_{10} &= \wp_{1811} \cdot \wp_{4391} \cdot \wp_{68626575961} \\
\wp_{1811} &= (\a^2 - \a) \z^3 + \a^2 + 1
\\
\wp_{4391} &= (-\a^2 + \a - 1) \z^3 + \z^2 - \a^2 \z - \a
\\
\wp_{68626575961} &=
(-2 \a^2 + \a + 10) \z^3 + (-\a^2 - 3 \a - 7) \z^2 + (-\a^2 + 2 \a + 5) \z + \a^2 - 3 \a - 5
\end{align*}
}
and for $(-2,3,7)$, respectively:
{\small
\begin{align*}
\varepsilon_{10} &= 
(2069226 \a^2 - 1143696 \a - 2044373) \z^3 + (308318 \a^2 + 1037807 \a + 609316) \z^2 + 
\\ & \quad\,\,
(1469403 \a^2 - 65443 \a - 886914) \z - 970534 \a^2 + 1744650 \a + 1872808
\\
\beta_{10} &= \wp_{2^{12}}^{1/2} \cdot \wp_{11^2} \cdot \wp_{971^2}
\cdot \wp_{1091} \cdot \wp_{1151} \\
\wp_{2^{6}} &= 2
\\
\wp_{11^2} &= (-\a^2 + \a) \z^3 + (-\a^2 + 1) \z + \a^2
\\
\wp_{971^2} &= \a \z^3 + (\a^2 + 2) \z^2 + (\a^2 - 1) \z - \a^2 + \a - 1
\\
\wp_{1091} &= -\a \z^3 + \a \z - \a^2 + \a
\\
\wp_{1151} &= (\a^2 + \a - 1) \z^2 + (-\a^2 + 1) \z + \a - 1
\end{align*}
}

Some 2 and 3-loop invariants for $5_2$ and $(-2,3,7)$ pretzel knots
are shown next.

{\small
$$
\begin{array}{|c|l|} \hline
k & S_{2,k} \qquad \text{for $5_2$} \\ \hline
1 & (245 - 242 \a - 33 \a^2)/2116 \\ \hline
2 & (6295 - 10303 \a - 1314 \a^2)/46552 \\ \hline
3 & (1763029 - 3730884 \a - 616974 \a^2)/11464488 + 
(727 + 40 \a - 52 \a^2) \zeta/6923
\\ \hline
4 & (198755261 - 468329838 \a - 88322976 \a^2)/1085609568 - 
\\
& 
(-144841 - 3059 \a + 3724 \a^2) \zeta/1966684
\\ \hline
5 & (1252389600136849 - 2036921357788788 \a - 
  291646682299854 \a^2)/3697084423961400 + 
\\ 
&
3 (109837198792 - 4170485943 \a + 4920447944 \a^2) \zeta/1339523342015 + 
\\ 
&
3(392592030863 - 20752850276 \a + 
    41177718597 \a^2) \zeta^2/6697616710075 + 
\\ 
&
3(-57107525462 - 19759788121 \a + 42866787232 \a^2) 
\zeta^3/6697616710075
\\ \hline
\end{array}
$$
$$
\begin{array}{|c|l|} \hline
k & S_{3,k} \qquad \text{for $5_2$} \\ \hline
1 & 3(18 - 155 \a + 155 \a^2)/24334\\ \hline
2 & 3(-70769 - 255956 \a + 319945 \a^2)/11777656 \\ \hline
3 & (-1863760571 - 9092540536 \a + 10659951670 \a^2)/59526487818
+
\\ 
&
(581674213 - 725755840 \a - 213728162 \a^2) \zeta/29763243909
\\ \hline
4 & 3(-1447363406795 - 7699225522158 \a + 
    9371835787629 \a^2)/88960456984688 -
\\ 
&
3(-6173681325057 + 7935320251722 \a + 1607053670497 \a^2) 
\zeta)/355841827938752
\\ \hline
5 & 3
(-10989752660217610311084459 -
59081913982949711575555062 \a + 
\\
&
      69563350243075727956792969 \a^2)/412694240274697972779851750 - 
\\
&
3(-4209383365964471973165111 + 
      5860219093140674277853192 \a + 
\\
&
586631030165980383508791 \a^2) 
\zeta/206347120137348986389925875
 + 
\\
&
9 (359260923564919009455273 - 
      2046639621559644326769101 \a + 
\\
&
      171017223371634425264447 \a^2) \zeta^2/412694240274697972779851750 + 
\\
&
3(-6713426920522807160021312 + 
      3867227919717696039743014 \a + 
\\
&
      2476626379634791382781767 \a^2) \zeta^3/412694240274697972779851750
\\ \hline
\end{array}
$$
}

{\small
$$
\begin{array}
{|c|l|} \hline
k & S_{2,k} \qquad \text{for $(-2,3,7)$ pretzel} \\ \hline
1 & (-73 - 1524 \a - 879 \a^2)/25392 \\ \hline
2 & (5213 - 6774 \a + 726 \a^2)/25392 \\ \hline
3 & (6428435 - 7198212 \a - 1601715 \a^2)/28413648 + 
(10598 - 6375 \a + 3506 \a^2) \zeta/51474 \\ \hline
4 & (1772576 - 2698227 \a -   1231152 \a^2)/9471216 + 
(11085 - 4012 \a + 1543 \a^2) \zeta/34316 \\ \hline
5 & (22745305769203 - 23958770711676 \a + 
  1292918125467 \a^2)/35941786270800 + 
\\ 
&  
(4515992099 - 1436127126 \a +     641928216 \a^2) \zeta/13022386330 + 
\\ 
&  
(30141870223 - 17414407586 \a + 
    9676608447 \a^2) \zeta^2/65111931650 + 
\\ 
&  
(14370066463 - 
    15291381996 \a + 9801845647 \a^2) \zeta^3/65111931650 \\ \hline
\end{array}
$$
$$
\begin{array}
{|c|l|} \hline
k & S_{3,k} \qquad \text{for $(-2,3,7)$ pretzel} \\ \hline
1 & (2099 - 2099 \a + 6874 \a^2)/778688 \\ \hline
2 & (-10438 + 8532 \a - 177 \a^2)/389344 \\ \hline
3 & (19141449113 - 148532821745 \a + 
  206516117210 \a^2)/2925128234304 + 
\\ \hline &  
(7427517757 - 10156808752 \a - 
    14120983571 \a^2) \zeta/731282058576 \\ \hline
4 & (181162947 - 969125569 \a + 
  2031947518 \a^2)/13542260344 + 
\\ \hline &  
(2348859343 - 5533235364 \a + 
    628243915 \a^2) \zeta/108338082752 \\ \hline
5 &
(-26527900336733230761869 + 1135819865279935909813 \a + 
\\
&
35348426895458618612714 \a^2)/312031884139098398776000 + 
\\ 
&  
(4131462185677760934998 - 5937844115855778532936 \a - 
\\
&
327921124596643988013 \a^2) \zeta/78007971034774599694000 + 
\\ 
&  
(-3615963624053978498519 + 2059165800970040528333 \a - 
\\
&
6722732999448794955676 \a^2) \zeta^2/78007971034774599694000 + 
\\ 
&  
(-7145964286998284204683 + 6060423272846646989631 \a 
\\
&
- 5019314438337630298992 \a^2) \zeta^3/78007971034774599694000
\\ \hline 
\end{array}
$$
}

\subsection{The $6_1$ knot}
\label{sub.61}

The $6_1$ knot is a hyperbolic knot with
volume $3.1639\dots$ with $4$ ideal tetrahedra and trace
field $F_{6_1}=\BQ(\a)$ where $\a=1.5041\dots - 1.2268\dots i$ is a root of 
$$
x^4 - 2 x^3 + x^2 + 3 x + 1 = 0
$$
$F_{6_1}$ is of type $[0,2]$ with discriminant $257$, a prime. We chose
to give the data for this knot because the Bloch group of its trace
field is a finitely generated abelian group of rank $2$.
The 1-loop invariant at $k=1$ and its norm is given by

\be
\label{eq.1loop61}
\begin{array}{|c|l|l|} \hline
\text{knot} & \tau_1^{-2} & N(\tau_1^{-2}) \\ \hline
6_1 & 7 \a^3 - 17 \a^2 + 17 \a + 12 & 257  \\ \hline
\end{array} 
\ee
The norm of the 1-loop of $6_1$ at level $k$ is given in~\eqref{t.61norm}.
{\tiny
\be
\label{t.61norm}
\begin{array}{|c|l|} \hline
k & N(\tau_k/\tau_1) \qquad \text{for $6_1$} \\ \hline
1 & 1 \\ \hline
2 & 29 \\ \hline
3 & 79 \cdot 373 \\ \hline
4 & 487057 \\ \hline
5 & 401 \cdot 8120801581 \\ \hline
6 & 4969 \cdot 33601 \\ \hline
7 & 2^3 \cdot 19013 \cdot 3957451 \cdot 33546226214089 \\ \hline
8 & 732209 \cdot 85423522285273 \\ \hline
9 & 2^12 \cdot 19^2 \cdot 199^2 \cdot 541 \cdot 12313999 \cdot 39491789023 
\\ \hline
10 & 100981 \cdot 317733001 \cdot 36502384021 \\ \hline
11 & 291418667 \cdot 3515449621583206989038092387793289595509816623 \\ \hline
12 & 157 \cdot 15086917 \cdot 479105929 \cdot 3349280377 \\ \hline
13 & 79 \cdot 117777271 \cdot 5870910773677 \cdot 
     644682638171983561196398860905544937015222089370889 \\ \hline
14 & 2^3 \cdot 1405219181759 \cdot 57474686640618078167230081699 \\ \hline
15 & 31^2 \cdot 2379691 \cdot 63360261033352141 \cdot 
     1042507380808009331327711940605725261 \\ \hline
16 & 196053041 \cdot 21917758321 \cdot 
2943442798173814595177658255884613139633 \\ \hline
18 & 2^6 \cdot 19^2 \cdot 4678492152497445991171 \cdot 
1135119536120342889490177 \\ \hline
20 & 32261 \cdot 500083848464103577816221055593641 \cdot 
     200729720160049090343996502563952161 \\ \hline
21 & 1009 \cdot 538727231341573 \cdot 
\\ &
69679537903457255216788492238561211259581901781788919921647100909271764443148059688148169 \\ \hline
22 & 23 \cdot 1304249 \cdot 17520427 \cdot 35064943 \cdot 662517155967701 \cdot 
     13980312643423978437421727 \cdot 653195100488320873699349233
\\ \hline
\end{array}
\ee
}

\noindent 
For $k=2$ 
we have
{\small
\begin{align*}
\varepsilon_{2} &= -\a^3 + 2 \a^2 - \a - 3
\\
\beta_{2} &= \wp_{29}  \\
\wp_{29} &= -4 \a^3 + 10 \a^2 - 8 \a - 7
\end{align*}
}

\noindent 
For $k=3$ and $\zeta=\zeta_3$
we have
{\small
\begin{align*}
\varepsilon_{3} &= 
(-4 \a^3 + 10 \a^2 - 9 \a - 8) \z - 3 \a^3 + 7 \a^2 - 5 \a - 8
\\
\beta_{3} &= \wp_{79} \cdot \wp_{373} \\
\wp_{79} &= (\a^3 - 2 \a^2 + \a + 2) \z + 3 \a^3 - 7 \a^2 + 5 \a + 6
\\
\wp_{373} &= (-2 \a^3 + 5 \a^2 - 5 \a - 3) \z - 3 \a^3 + 7 \a^2 - 7 \a - 4
\end{align*}
}

\noindent 
For $k=4$ and $\zeta=\zeta_4$ we have
{\small
\begin{align*}
\varepsilon_{4} &= 
(\a^3 - 3 \a^2 + 3 \a) \z + 4 \a^3 - 10 \a^2 + 8 \a + 10
\\
\beta_{4} &= \wp_{487057}  \\
\wp_{487057} &=
(2 \a^3 - 4 \a^2 + 2 \a + 4) \z + \a^3 - \a^2 + 3 \a - 2
\end{align*}
}

\noindent 
For $k=5$ and $\zeta=\zeta_5$ we have
{\small
\begin{align*}
\varepsilon_{5} &= 
(-2 \a^2 + 6 \a + 24) \z^3 + (-20 \a^3 + 56 \a^2 - 38 \a - 12) \z^2 + 
\\ & \quad\,\,
(-40 \a^3 + 114 \a^2 - 92 \a - 68) \z - 18 \a^3 + 56 \a^2 - 49 \a - 40
\\
\beta_{5} &= \wp_{401} \cdot \wp_{8120801581}  \\
\wp_{401} &=
(-\a^3 + 2 \a^2 - \a - 2) \z^3 + \z^2 + (\a^3 - 3 \a^2 + 3 \a + 2) \z
\\
\wp_{8120801581} &=
(\a + 3) \z^3 + (-5 \a^3 + 11 \a^2 - 9 \a - 5) \z^2 + (-2 \a^3 + 3 \a^2 - 3 \a - 1) \z - \a^2 + 4
\end{align*}
}

\noindent 
For $k=6$ and $\zeta=\zeta_6$ we have
{\small
\begin{align*}
\varepsilon_{6} &= 
(3 \a^3 - 8 \a^2 + 7 \a + 5) \z - 3 \a^3 + 7 \a^2 - 5 \a - 6
\\
\beta_{6} &= \wp_{4969} \cdot \wp_{33601}  \\
\wp_{4969} &= (3 \a^3 - 7 \a^2 + 8 \a + 4) \z + 1
\\
\wp_{33601} &= (4 \a^3 - 10 \a^2 + 8 \a + 5) \z + \a^3 - 3 \a^2 + 4 \a + 2
\end{align*}
}

Some 2 and 3-loop invariants are shown next. 
{\tiny
$$
\begin{array}{|c|l|} \hline
k & S_{2,k} \qquad \text{for $6_1$} \\ \hline
1 & (-178515 - 946382 \a + 924836 \a^2 - 371920 \a^3)/1585176 \\ \hline
2 & (-27011582 - 51129989 \a + 48845639 \a^2 - 19497370 \a^3)/45970104 \\ \hline
3 & (-82893368809 - 117384982993 \a + 115430695442 \a^2 - 
  47280180216 \a^3)/70065571788 + 
\\ & (1706191 - 1154600 \a + 
    2708170 \a^2 - 1385605 \a^3) \zeta/22719057 \\ \hline
4 & (1/3088284268128)(-4950930619209 - 7026286049126 \a + 
   7165813225694 \a^2 - 
\\ &
   2954092842556 \a^3) + ((84879497 - 121998463 \a + 149562867 \a^2 - 
\\ &
    55782966 \a^3) \zeta)/500694596 \\ \hline
5 & -103464360336910543873 - 188649056185634232247 \a + 
    173804121553360109686 \a^2 -
\\ &
 66971292202517629952 \a^3)/
  64525410081903320700 + (2813833153341350 - 
\\ & 
62305691106986 \a - 
      557364703389415 \a^2 + 377823446091675 \a^3) \zeta/
  4184527242665585 + 
\\ & 
(13304890388975226 + 6297147216121499 \a - 
      10590487560881967 \a^2 + 
\\ & 
4980152420171024 \a^3) \zeta^2/
  20922636213327925 + (4532417943052961 + 6683187696077234 \a - 
\\ &
      9628594457667602 \a^2 + 4438295075710969 \a^3) \zeta^3/
  20922636213327925 \\ \hline
\end{array}
$$
$$
\begin{array}{|c|l|} \hline
k & S_{3,k} \qquad \text{for $6_1$} \\ \hline
1 & (-2772972 - 2244430 \a + 2833463 \a^2 - 1140832 \a^3)/33949186 \\ \hline
2 & (-32774690022 - 17111505319 \a + 26321905652 \a^2 - 
 10527251164 \a^3)/114205061704 \\ \hline
3 & (-1598504997001206261 - 909085206892628307 \a + 
    1322686345008572948 \a^2 - 
\\ & 
540917115639525443 \a^3)/
  2387735578783745874 + (-340987970089137309 - 
\\ &      593382515118577161 \a + 540555632185247860 \a^2 - 
\\ &
      224218760661580090 \a^3) \zeta/2387735578783745874 \\ \hline
4 & (-76552043703957527182 - 43852642902836424033 \a + 
    62680976630422417186 \a^2 - 
\\ &
25715623922859379240 \a^3)/
  64428635165146026512 - 3 (40174169918962174465 + 
\\ &
      74704051006678295591 \a - 69345923760309927344 \a^2 + 
\\ &
      27498092656227102870 \a^3) \zeta/257714540660584106048 \\ \hline
5 & (-83535030268880547833711035882206548 - 
    57630500922078935770505946948391332 \a + 
\\ &
    79407230942955753160123180497661371 \a^2 - 
\\ &   30884930225756906416238335746759432 \a^3)/
  45001389388648840291462688518018250 + 
\\ & (-16829652415100927830509785657370971 - 
      36833773286121610403780003873363084 \a + 
\\ &
      34528306662808298407142994334953867 \a^2 - 
\\ &      13138908536274467739177531307979864 \a^3) \zeta/
  45001389388648840291462688518018250 + 
\\ & (1887348067005309217790728173337463 + 
      8201430020783561588938132296311242 \a - 
\\ &      5410900283083316911728968138281266 \a^2 + 
      2484658725916747927696004662854207 \a^3) \zeta^2/
\\ &  45001389388648840291462688518018250 + 
(14884459939384051281789536921780086 + 
\\ &      35638731679608054276763036584609939 \a - 
      30236970444383524674457588584916012 \a^2 + 
\\ &      12454010292496128569264912949291219 \a^3) \zeta^3/
  45001389388648840291462688518018250 \\ \hline
\end{array}
$$
}

\subsection{The $(-2,3,-3)$ and the $(-2,3,9)$ partner pretzel knots}
\label{sub.239}

The $(-2,3,9)$ and the mirror of the $(-2,3,-3)$ pretzel knots (the latter
is also known as the $8_{20}$ knot) are partners. They can both be assembled 
from the same set of ideal tetrahedra. It follows that they have equal
volume $4.1249\dots$ and equal elements of the Bloch group. 
They also have equal trace fields $F_{(-2,3,-3)}=
F_{(-2,3,9)}=\BQ(\a)$ where $\a=0.4425\dots - 0.4544\dots i$ is a root
of
$$
x^5 - x^4 + x^3 + 2 x^2 - 2 x + 1 =0 
$$
This field is of type $[1,2]$ with discriminant $2^3 \cdot 733$.
The 1-loop invariant at $k=1$ and its norm is given by

\be
\label{eq.1loop239}
\begin{array}{|c|l|l|} \hline
\text{knot} & \tau_1^{-2} & N(\tau_1^{-2}) \\ \hline
(-2,3,-3) &  -10 \a^4 + 8 \a^3 - 7 \a^2 - 22 \a + 13 & - 2^4 \cdot 733 
\\ \hline
(-2,3,9) & -4 \a^4 + 10 \a^3 - 10 \a^2 + 2 \a + 14 & -2^7 \cdot 733 
\\ \hline
\end{array} 
\ee

The norm of the 1-loop of the $(-2,3,-3)$ and $(-2,3,9)$ pretzel knots
at level $k$ is given in~\eqref{t.233norm} and~\eqref{t.239norm}
respectively.
{\tiny
\be
\label{t.233norm}
\begin{array}{|c|l|} \hline
k & N(\tau_k/\tau_1) \qquad \text{for $(-2,3,-3)$ pretzel} \\ \hline
1 & 1 \\ \hline
2 & 9\cdot \sqrt{2} \\ \hline
3 & 86677 \\ \hline
4 & 2^2 \cdot 389 \cdot 829 \\ \hline
5 & 251 \cdot 3701 \cdot 5641 \cdot 9573881 \\ \hline
6 & 2 \cdot 3^2 \cdot 73 \cdot 1675763533 \\ \hline
7 & 21059216779259 \cdot 15637926099144015661 \\ \hline
8 & 2^8 \cdot 1677121 \cdot 2821611376969577 \\ \hline
9 & 37 \cdot 288361 \cdot 16887730311458362485922054098785491 \\ \hline
10 & 2^2 \cdot 6451 \cdot 765151 \cdot 2036899317566108665824611 \\ \hline
11 & 572683 \cdot 15481222769 \cdot 123058773843133908627743611 \cdot 
     40590314050385646643724337053081 \\ \hline
12 & 2^4 \cdot 7^2 \cdot 11701 \cdot 570178703041 \cdot 76017401206533083977 
\\ \hline
13 & 3^3 \cdot 3121 \cdot 8581 \cdot \\ &
5208780692011162885806751823435154606807560938151916182486066554111775765097437387670769 \\ \hline
14 & 2^3 \cdot 29 \cdot 883 \cdot 95890797076684070930617 \cdot 
     196704656196706336391779227757264369 \\ \hline
15 & 2731 \cdot 84871 \cdot 517081 \cdot 73175750117941351 \cdot 
2791635002919906087031 \cdot 
     49318837138663878429931849195141 \\ \hline
16 & 2^{16} \cdot 337 \cdot 54673 \cdot 55181281 \cdot 16869371249354588848817 
\cdot 
     2300418425808890616155725510116534231121
\\ \hline
\end{array}
\ee
}

{\tiny
\be
\label{t.239norm}
\begin{array}{|c|l|} \hline
k & N(\tau_k/\tau_1) \qquad \text{for $(-2,3,9)$ pretzel} \\ \hline
1 & 1 \\ \hline
2 & 2^2 \\ \hline
3 & 18217 \\ \hline
4 & 2^5 \cdot 3^2 \cdot 29 \cdot 101 \\ \hline
5 & 31 \cdot 28541 \cdot 1731399041 \\ \hline
6 & 2^4 \cdot 31 \cdot 11059 \cdot 171043 \\ \hline
7 & 43 \cdot 197^2 \cdot 218454864083787040860053 \\ \hline
8 & 2^14 \cdot 241 \cdot 3361 \cdot 1003193 \cdot 15946313 \\ \hline
9 & 17594311532167603 \cdot 50545284538200619535209 \\ \hline
10 & 2^8 \cdot 31^3 \cdot 101 \cdot 74098513175515361398321 \\ \hline
11 & 23^2 \cdot 1917210263 \cdot 869440556615693617955386097 \cdot 
     3945088199088552145275994613 \\ \hline
12 & 2^10 \cdot 3^2 \cdot 61 \cdot 73 \cdot 229 \cdot 
483950196831635581064375269 \\ \hline
13 & 394369 \cdot 1817999 \cdot 910838088184909 \cdot \\ &
     1305464531078495668738541122232633072902748130522596602480377557 \\ \hline
14 & 2^12 \cdot 23031231430410673 \cdot 6254905477428365514627650788018899766459 \\ \hline
15 & 61 \cdot 96331 \cdot 4470070545691 \cdot 
     144159618930221245901711143825031366430385501494972724861589731 \\ \hline
16 & 2^28 \cdot 7^2 \cdot 764993 \cdot 13057776577 \cdot 
     3859919412481173535559894253284600362010320636824481
\\ \hline
\end{array}
\ee
}

Next, we give some sample computations of the 
factorization~\eqref{eq.phi0factor}. 

\noindent 
For $k=2$ 
we have for $(-2,3,-3)$
{\small
\begin{align*}
\varepsilon_{2} &= \a^3
\\
\beta_{3} &= \wp_{2}^{1/2} \cdot \wp_{3^2}  \\
\wp_{2} &= -\a^4 + \a^3 - \a^2 - 2 \a + 1
\\
\wp_{3^2} &= \a^3 + \a + 1
\end{align*}
}
and for $(-2,3,9)$, respectively:
{\small
\begin{align*}
\varepsilon_{2} &= \a^4 - \a^2 + \a
\\
\beta_{2} &= \wp_{2^3}^{1/2} \cdot \wp_{2}^{1/2}  \\
\wp_{2^3} &= \a^4 - \a^3 + 3 \a - 2
\\
\wp_{2} &= -\a^4 + \a^3 - \a^2 - 2 \a + 1
\end{align*}
}

\noindent 
For $k=3$ and $\zeta=\zeta_3$
we have for $(-2,3,-3)$
{\small
\begin{align*}
\varepsilon_3 &= (\a^4 + \a^2 + 2 \a + 1) \z + \a^4 + 3 \a - 1
\\
\beta_3 &= \wp_{86677}  \\
\wp_{86677} &= (3 \a^4 - \a^3 + 8 \a - 1) \z + \a^4 + 3 \a
\end{align*}
}
and for $(-2,3,9)$, respectively:
{\small
\begin{align*}
\varepsilon_3 &= 
(\a^4 + \a^2 + 3 \a) \z + \a^4 - \a^3 + \a^2 + 2 \a - 3
\\
\beta_3 &= \wp_{18217}  \\
\wp_{18217} &= (\a^2 + \a + 2) \z + \a^2 + 2
\end{align*}
}

\noindent 
For $k=4$ and $\zeta=\zeta_4$ we have for $(-2,3,-3)$
{\small
\begin{align*}
\varepsilon_{4} &= 
(97/2 \a^4 - 36 \a^3 + 55/2 \a^2 + 211/2 \a - 151/2) \z - 
\\ & \quad\,\,
37/2 \a^4 - 13 \a^3 - 15/2 \a^2 - 107/2 \a - 75/2
\\
\beta_{4} &= \wp_{2}^2 \cdot \wp_{389} \cdot \wp_{829} \\
\wp_{2} &=
(-1/2 \a^4 - 1/2 \a^2 - 1/2 \a + 1/2) \z - 1/2 \a^4 - 1/2 \a^2 - 3/2 \a + 1/2
\\
\wp_{389} &=
(-\a^4 + \a^3 - 3 \a + 2) \z - \a^4 + \a^3 - \a^2 - 3 \a + 2
\\
\wp_{829} &=
(-3/2 \a^4 + \a^3 - 3/2 \a^2 - 7/2 \a + 3/2) \z - 1/2 \a^4 + 1/2 \a^2 - 3/2 \a + 3/2
\end{align*}
}
and for $(-2,3,9)$, respectively:
{\small
\begin{align*}
\varepsilon_{4} &= 
(-43/2 \a^4 + 40 \a^3 - 105/2 \a^2 + 29/2 \a + 31/2) \z + 
\\ & \quad\,\,
5/2 \a^4 + 22 \a^3 - 73/2 \a^2 + 127/2 \a - 35/2
\\
\beta_{4} &= \wp_{2^3} \cdot \wp_{2}^2 \cdot \wp_{3^2}
\cdot \wp_{29} \cdot \wp_{101} \\
\wp_{2^3} &=
(-1/2 \a^4 - 1/2 \a^2 - 3/2 \a + 1/2) \z + 1/2 \a^4 + 1/2 \a^2 + 3/2 \a - 1/2
\\
\wp_{2} &=
(-1/2 \a^4 - 1/2 \a^2 - 1/2 \a + 1/2) \z - 1/2 \a^4 - 1/2 \a^2 - 3/2 \a + 1/2
\\
\wp_{3^2} &=
(-1/2 \a^4 - 1/2 \a^2 - 3/2 \a - 1/2) \z - 1/2 \a^4 - 1/2 \a^2 - 1/2 \a + 1/2
\\
\wp_{29} &= (\a^4 + \a^2 + 2 \a) \z - \a
\\
\wp_{101} &= (\a^4 - \a^3 + 2 \a - 2) \z + \a^4 - \a^3 + \a^2 + 2 \a - 2
\end{align*}
}

\noindent 
For $k=5$ and $\zeta=\zeta_5$ we have for $(-2,3,-3)$
{\small
\begin{align*}
\varepsilon_{5} &= 
(6 \a^4 + 5 \a^3 + 2 \a^2 - 2 \a - 9) \z^3 + (-13 \a^4 + 30 \a^3 + 9 \a^2 - 49 \a + 7) \z^2 + 
\\ & \quad\,\,
(-32 \a^4 + 30 \a^3 + 11 \a^2 - 66 \a + 23) \z - 31 \a^4 + 19 \a^3 + 10 \a^2 - 48 \a + 25
\\
\beta_{5} &= \wp_{251} \cdot \wp_{3701} \cdot \wp_{5641}
\cdot \wp_{9573881} \\
\wp_{251} &=
(\a^4 - \a^3 + \a^2 + 2 \a - 1) \z^2 + \a \z
\\
\wp_{3701} &=
(-\a^4 - 2 \a) \z^3 + (\a^4 - \a^3 + \a^2 + 2 \a - 2) \z^2 + \a
\\
\wp_{5641} &=
(\a^3 + \a + 1) \z^3 + (\a^4 + 3 \a - 1) \z^2 + (\a^4 + 3 \a - 1) \z + \a^4 + 3 \a - 1
\\
\wp_{9573881} &=
\z^3 + (\a^3 + 2) \z^2 + (-\a^4 + \a^3 - 2 \a) \z - \a^4 + \a^3 - 2 \a + 2
\end{align*}
}
and for $(-2,3,9)$, respectively:
{\small
\begin{align*}
\varepsilon_{5} &= 
(222 \a^4 - 1294 \a^3 - 585 \a^2 + 990 \a - 662) \z^3 + 
\\ & \quad\,\,
(1179 \a^4 - 1708 \a^3 - 1460 \a^2 + 1910 \a - 994) \z^2 + 
\\ & \quad\,\,
(1551 \a^4 - 677 \a^3 - 1420 \a^2 + 1495 \a - 541) \z + 
\\ & \quad\,\,
819 \a^4 + 379 \a^3 - 515 \a^2 + 312 \a + 74
\\
\beta_{5} &= \wp_{31} \cdot \wp_{28541} \cdot \wp_{1731399041} \\
\wp_{31} &=
(-\a^4 + \a^3 - \a^2 - 2 \a + 2) \z^3 + (-\a^4 + \a^3 - \a^2 - 2 \a + 2) \z^2 + 
\\ & \quad\,\,
(-\a^4 + \a^3 - \a^2 - 2 \a + 1) \z - \a^4 + \a^3 - \a^2 - 2 \a + 2
\\
\wp_{28541} &=
(\a^4 - \a^3 + \a^2 + \a - 2) \z^3 + (\a^4 + 2 \a) \z + 2 \a^4 - \a^3 + \a^2 + 5 \a - 2
\\
\wp_{1731399041} &=
-\a^2 \z^3 + (2 \a^4 - \a^2 + 3 \a) \z^2 + (3 \a^4 + \a^3 - \a^2 + 4 \a) \z + 3 \a^3 - \a^2 - \a + 3
\end{align*}
}

\noindent 
For $k=6$ and $\zeta=\zeta_6)$ we have for $(-2,3,-3)$
{\small
\begin{align*}
\varepsilon_{6} &= 
(20 \a^4 - 2 \a^3 + 29 \a^2 + 21 \a + 35) \z - 71 \a^4 + 89 \a^3 - 144 \a^2 - 9 \a + 5
\\
\beta_{6} &= \wp_{2^2}^{1/2} \cdot \wp_{3^2} 
\cdot \wp_{73} \cdot \wp_{1675763533} \\
\wp_{2^2} &= (-\a^4 + \a^3 - \a^2 - 2 \a + 1) \z
\\
\wp_{3^2} &= \z + \a^4 - \a^3 + 2 \a - 2
\\
\wp_{73} &=
(\a^4 - \a^3 + 2 \a - 1) \z + \a^4 + 2 \a - 1
\\
\wp_{1675763533} &=
(-5 \a^4 + \a^3 + 2 \a^2 - 10 \a + 3) \z + 5 \a^4 + \a^3 + \a^2 + 14 \a + 2
\end{align*}
}
and for $(-2,3,9)$, respectively:
{\small
\begin{align*}
\varepsilon_{6} &= 
(-30 \a^4 + 4 \a^3 + 27 \a^2 - 21 \a + 5) \z + 11 \a^4 + 34 \a^3 + \a^2 - 17 \a + 14
\\
\beta_{6} &= \wp_{2^6}^{1/2} \cdot \wp_{2^2}^{1/2} \cdot \wp_{31}
\cdot \wp_{11059} \cdot \wp_{171043} \\
\wp_{2^6} &= (\a^4 - \a^3 + 3 \a - 2) \z - \a^4 + \a^3 - 3 \a + 2
\\
\wp_{2^2} &= (-\a^4 + \a^3 - \a^2 - 2 \a + 1) \z
\\
\wp_{31} &= (\a + 1) \z - \a
\\
\wp_{11059} &= (-\a + 2) \z + \a^3 - \a^2 + 1
\\
\wp_{171043} &=
(-2 \a^4 + 2 \a^3 - 2 \a^2 - 3 \a + 2) \z + 5 \a^4 - 4 \a^3 + 3 \a^2 + 9 \a - 6
\end{align*}
}

\subsection{The $9_{12}$ knot}
\label{sub.912}

The $9_{12}$ knot has volume $8.3664\dots$ with $10$ ideal tetrahedra and trace
field $F_{9_12}=\BQ(\a)$ where 
$\a=-0.06265158\dots + i \,1.24990458\dots$ is a root of 
\begin{align*}
&x^{17} - 8 x^{16} + 32 x^{15} - 89 x^{14} + 195 x^{13} - 353 x^{12} + 542 x^{11} - 
719 x^{10} + 834 x^9  \\
& \qquad - 851 x^8 + 764 x^7 - 605 x^6+ 421 x^5 - 
253 x^4 + 130 x^3 - 55 x^2 + 18 x - 3=0
\end{align*}
$F_{9_{12}}$ is of type $[1,8]$ with discriminant $3 \cdot 298171 \cdot 5210119
\cdot 156953399$. We chose this final example because of the complexity of
the ideal triangulation, and the complexity of its trace field.
The 1-loop invariant at $k=1$ and its norm is given by

\be
\label{eq.1loop912}
\begin{array}{|c|l|l|} \hline
\text{knot} & \tau_1^{-2} & N(\tau_1^{-2}) \\ \hline
9_{12} &  59 \a^{16} + 40 \a^{15} - 14 \a^{14} + 12 \a^{13} - 164 \a^{12}
- 82 \a^{11} + & 
 3 \cdot 298171 \cdot 5210119 \cdot 
 \\ &
107 \a^{10} - 186 \a^9 - 55 \a^8 + 356 \a^7 - 387 \a^6 
- 410 \a^5 + & 156953399 \\ & 342 \a^4 + 207 \a^3 - 117 \a^2 - 68 \a - 13 &
\\ \hline
\end{array} 
\ee
The norm of the 1-loop of $9_{12}$ at level $k$ is given in~\eqref{t.912norm}.
{\tiny
\be
\label{t.912norm}
\begin{array}{|c|l|} \hline
k & N(\tau_k/\tau_1) \\ \hline
1 & 1 \\ \hline
2 & 175013 \cdot 139320586381  \\ \hline
3 & 2^6 \cdot 3 \cdot 317089 \cdot 618610957 \cdot 16597704247 
\cdot 17781027987117308670607579
\\ \hline
4 & 89^2 \cdot 193 \cdot 
    113664060425758850100362844843557553491441831726215669353830969 
\\ \hline
\end{array}
\ee
}

\noindent 
For $k=2$ and $\zeta=e(1/2)$ we have
{\small
\begin{align*}
\varepsilon_{2} &= 
4 \a^{16} - 2 \a^{13} - 8 \a^{12} - 2 \a^{11} + 8 \a^{10} - 13 \a^9 + 4 \a^8 + 16 \a^7 - 24 \a^6 - 20 \a^5 +
\\ & \quad\,\,
28 \a^4 + 12 \a^3 - 16 \a^2 - 4 \a + 4
\\
\beta_{2} &= \wp_{175013} \cdot \wp_{139320586381} \\
\wp_{175013} &=
-9 \a^{16} + 12 \a^{15} - 2 \a^{14} + 5 \a^{13} + 15 \a^{12} - 23 \a^{11} 
- 17 \a^{10} + 50 \a^9 - 57 \a^8 - 
\\ & \quad\,\,
8 \a^7 + 98 \a^6 - 49 \a^5 - 82 \a^4 + 52 \a^3 + 34 \a^2 - 17 \a - 8
\\
\wp_{139320586381} &=
26 \a^{16} - 5 \a^{15} - 4 \a^{14} - 11 \a^{13} - 53 \a^{12} + 2 \a^{11} 
+ 61 \a^{10} - 97 \a^9 + 42 \a^8 + 
\\ & \quad\,\,
110 \a^7 - 193 \a^6 - 94 \a^5 + 225 \a^4 + 45 \a^3 - 124 \a^2 - 12 \a + 32
\end{align*}
}





\section*{Acknowledgements}

The work of SG is supported by NSF grant DMS-14-06419. This paper was primarily completed
while TD was a long-term member at the Institute for the Advanced Study, supported by the Friends of the Institute for Advanced Study, in part by DOE grant DE-FG02-90ER40542, and in part by 
ERC Starting Grant no. 335739 {\em Quantum fields and knot homologies}, 
funded by the European Research Council under the European Union's 
Seventh Framework Programme. TD is currently supported by the Perimeter Institute for Theoretical Physics; research at Perimeter Institute is supported by the Government of Canada through Industry Canada and by the Province of Ontario through the Ministry of Economic Development and Innovation.

The authors wish to thank Don Zagier for many enlightening conversations.
This paper could not have been written without Zagier's guidance,  
encouragement, and his generous sharing of ideas. The authors also thank
Sergei Gukov and Edward Witten for very helpful comments and suggestions.

\bibliographystyle{hamsalpha}
\bibliography{biblio}
\end{document}